\documentclass[11pt]{amsart}
\usepackage{latexsym,amscd,amssymb, graphicx, color, amsthm, bm, amsmath, cancel, enumitem}  
\usepackage{tikz}
\usetikzlibrary{calc} 
\usepackage[margin=1in]{geometry}
\usepackage{hyperref}
\usepackage[all,cmtip]{xy}
\usepackage{ytableau}
\usepackage[colorinlistoftodos]{todonotes}

\numberwithin{equation}{section}

\newtheorem{theorem}{Theorem}[section]
\newtheorem{proposition}[theorem]{Proposition}
\newtheorem{corollary}[theorem]{Corollary}
\newtheorem{lemma}[theorem]{Lemma}

\newtheorem{problem}[theorem]{Problem}
\newtheorem{example}[theorem]{Example}
\newtheorem{remark}[theorem]{Remark}

\makeatletter
\newtheorem*{rep@theorem}{\rep@title}
\newcommand{\newreptheorem}[2]{%
\newenvironment{rep#1}[1]{%
 \def\rep@title{#2 \ref{##1}}%
 \begin{rep@theorem}}%
 {\end{rep@theorem}}}
\makeatother
\theoremstyle{definition}
\newtheorem{definition}[theorem]{Definition}

\newreptheorem{lemma}{Lemma}
\newreptheorem{theorem}{Theorem}

\def\cocoa{{\hbox{\rm C\kern-.13em o\kern-.07em C\kern-.13em o\kern-.15em A}}}

\newcommand{\III}{{\mathcal {I}}}

\newcommand{\des}{{\mathrm {des}}}

\newcommand{\maj}{{\mathrm {maj}}}

\newcommand{\symm}{{\mathfrak{S}}}

\newcommand{\CC}{{\mathbb {C}}}
\newcommand{\EE}{{\mathbb {E}}}
\newcommand{\NN}{{\mathbb {N}}}
\newcommand{\PP}{{\mathbb {P}}}
\newcommand{\RR}{{\mathbb {R}}}

\newcommand{\ZZ}{{\mathbb {Z}}}

\newcommand{\VV}{{\mathbf{V}}}
\newcommand{\CCC}{{\mathcal{C}}}

\newcommand{\AAA}{{\mathcal{A}}}

\newcommand{\MMM}{{\mathcal{M}}}

\newcommand{\VVV}{{\mathcal{V}}}

\newcommand{\one}{{\mathbf{1}}}

\newcommand{\exc}{{\mathrm {exc}}}

\begin{document}

\title[Moments of permutation statistics by cycle type]
{Moments of permutation statistics by cycle type}

\author[Zachary Hamaker]{Zachary Hamaker}
\author[Brendon Rhoades]{Brendon Rhoades}

\address
{Department of Mathematics \newline \indent
University of Florida \newline \indent
Gainesville, FL, 32611, USA}
\email{zhamaker@ufl.edu}

\address
{Department of Mathematics \newline \indent
University of California, San Diego \newline \indent
La Jolla, CA, 92093, USA}
\email{bprhoades@ucsd.edu}

\begin{abstract}
Beginning with work of Zeilberger on classical pattern counts, there are a variety of structural results for moments of permutation statistics applied to random permutations.
Using tools from representation theory, Gaetz and Ryba generalized Zeilberger's results to uniformly random permutations of a given cycle type.
We introduce \emph{regular statistics} and characterize their moments for all cycle types, generalizing all results in this literature that we are aware of.
Our approach splits into two steps: first characterize such statistics as linear combinations of indicator functions for partial permutations, then identifying the moments of such indicators. 
As an application, we show that many regular statistics exhibit a law of large numbers depending only on fixed point counts and a similar variance property that depends also on two--cycle counts.

These results first appeared in~\href{https://arxiv.org/abs/2206.06567}{arXiv:2206.06567}, which is no longer intended for publication.
Our original proof of the moment result for indicators of partial permutations relied on representation theory and symmetric functions.
A referee generously shared a combinatorial argument, allowing us to give a self-contained treatment of these results that does not rely on representation theory.
\end{abstract}

\maketitle

\section{Introduction}
\label{sec:Introduction}
A permutation statistic is a function that is defined uniformly on all symmetric groups simultaneously. 
Fundamental questions in a wide variety of computational fields including computer science and statistics are best understood by characterizing the typical behavior of permutation statistics.
We introduce a novel family of permutation statistics called \emph{regular statistics} and show how to understand their asymptotic properties when applied to uniformly random permutations of a given cycle type.
Regular statistics can be thought of as weighted counts of small patterns and include (bivincular) pattern counts, cycle counts and a variety of other well known statistics.


Let $\symm_n$ be the group of permutations of $[n] = \{1,2,\dots,n\}$.
A natural notion of `pattern' in the permutation $\pi \in \symm_n$  is as an instance where $\pi$ maps certain distinct positions $I = (i_1,\dots,i_k)$ to distinct values $J = (j_1,\dots,j_k)$.
Pattern counting statistics then count patterns with various restrictions on how these positions and values should relate.
We call the pair $(I,J)$ a \emph{partial permutation} of \emph{size} $k$.
Define
\begin{equation}
    \label{eq:one_IJ}
    1_{ij}(\pi) = \begin{cases}
        1 & \pi(i) = j\\
        0 & \text{else}
    \end{cases}
    \quad \text{and} \quad 1_{IJ} = 1_{i_1 j_1} \cdot 1_{i_2 j_2} \cdot \ldots \cdot 1_{i_k j_k}
\end{equation}
with the latter products taken pointwise.
Then weighted pattern counts can be expressed as linear combinations of $1_{IJ}$'s.
Much as permutations have a cycle type, the partial permutation $(I,J)$ has a \emph{cycle--path type} $(\mu,\nu)$ where $\mu$ and $\nu$ are the partitions of cycle and path lengths, respectively, when viewing $(I,J)$ as a directed graph.

For $\lambda$ a partition of $n$, let $K_\lambda$ be the set of permutations in $\symm_n$ whose cycle type is $\lambda$.
Also, let $m_i(\lambda)$ be the number of parts in $\lambda$ of size $i$.
For $f$ a permutation statistic, define $\mathbb{E}_\lambda[f]$ as the expected value of $f$ applied to a uniformly random element of $K_\lambda$.
Our main tool for analysis is:

\begin{theorem}
    \label{t:moment}
    Let $(I,J)$ be a partial permutation of size $k$ with cycle--path type $(\mu,\nu)$ and $|I \cup J| = m$.
    Then there is a polynomial $f_{(\mu,\nu)} \in \ZZ[n,m_1,\dots,m_k]$ of degree $k$ where $\deg n = 1$ and $\deg m_i = i$ so that for any partition $\lambda$
\begin{equation}
(n)_m \cdot \EE_\lambda[1_{IJ}] = f_{(\mu,\nu)}(n,m_1(\lambda),m_2(\lambda),\dots,m_k(\lambda)).
\label{eq:moment}    
\end{equation}
\end{theorem}

From representation theoretic considerations, it is relatively straightforward to prove $\EE_\lambda[1_{IJ}]$ depends only $n$ and $m_1(\lambda),\dots,m_k(\lambda)$.
With this perspective in mind, the surprising part of this result is the polynomiality, which can be interpreted as a representation stability result.

Our original proof of Theorem~\ref{t:moment} appears in~\cite{HR}, which is no longer intended for publication.
There, we interpreted $\EE_{\lambda}[1_{IJ}]$ as the evaluation of a class function on an element of $K_\lambda$, giving us access to tools from representation theory and symmetric function theory.
Using these methods, we gave a combinatorial formula for $\EE_{\lambda}[1_{IJ}]$ from which we derived Theorem~\ref{t:moment}.
In that setting, the unusual grading for $f_{(\mu,\nu)}$ is natural.
However, a referee sketched a direct combinatorial argument based on M\"obius inversion in the set partition lattice, which is both elementary and far easier.
We present it in Section~\ref{sec:Moments}, with our original representation theoretic approach presented in a separate paper~\cite{HR2}.
By the theory of character polynomials, Theorem~\ref{t:moment} can also be interpreted in terms of irreducible characters or symmetric functions.
See~\cite[\S 7]{HR2} for further details.

Using Theorem~\ref{t:moment}, the moments of linear combinations of $1_{IJ}$'s can be easily understood if they can be grouped by cycle--path type.
Our definition of regular statistics, which we now present, is designed to facilitate such groupings.

We say the partial permutation $(U,V)$ is a \emph{packed} if $U \cup V = [m]$ with $m$ a positive integer.
For $S = \{s_1 < \dots < s_m\} \subseteq \NN$ with $U = (u_1,\dots,u_k)$ and $V = (v_1,\dots,v_k)$, let $S(U) = (s_{u_1},\dots,s_{u_k})$ and $S(V) = (s_{v_1},\dots,s_{v_k})$.
Note every partial permutation $(I,J)$ is $S(U,V)$ for some packed partial permutation $(U,V)$ and set $S$.
Lastly, for $C \subseteq [m-1]$ let $\binom{\NN}{m}_C$ be the set of $m$--element subsets $S = \{s_1 < \dots < s_m\}$ of $\NN$ so that for $i \in C$ we have $s_{i+1} = s_i + 1$.
\begin{definition}
	\label{d:intro-regular}

For $(U,V)$ a packed partial permutation with $U \cup V = [m]$, $C \subseteq [m-1]$ and $f \in \CC[x_1,\dots,x_m]$, the associated \emph{constrained translate} is
\[
T^f_{(U,V),C} = \sum_{L = \{\ell_1 < \dots < \ell_m\} \in \binom{\NN}{m}_C} f(\ell_1,\dots,\ell_m) \cdot 1_{L(U)\,L(V)}.
\]
We say $T^f_{(U,V),C}$ has \emph{size} $k$, \emph{shift} $|C|$ and \emph{power} $k + \deg f - |C|$.

A \emph{regular statistic} is a linear combination of constrained translates.
Its size, shift and power are defined analagously (see Definition~\ref{def:regular-size} for details).
\end{definition}

Regular statistics generalize several well-known families of statistics, including classical pattern counts, vincular pattern counts and the bivincular statistics introduced in~\cite{DK}.
For example, the \emph{major index} $\maj$ defined by 
\begin{equation}
\label{eq:maj}
\maj(\pi) = \sum_{k: \pi_k > \pi_{k+1}} k = \sum_{i,j<k} i \cdot 1_{(i,i{+}1)(k,j)}
\end{equation}
is a bivincular statistic.
To see it is regular, we realize it using vincular translates as
\begin{align*}
\maj
 =&\ \  T^{x_1}_{(12)(21),\{1\}} + T^{x_1}_{(12)(31),\{1\}}+ T^{x_1}_{(12)(32),\{1\}} + T^{x_2}_{(23)(21),\{2\}}\\ &+ T^{x_2}_{(23)(31),\{2\}} +
T^{x_1}_{(12)(43),\{1\}}
+ T^{x_2}_{(23)(41),\{2\}} + T^{x_3}_{(34)(12),\{3\}}.
\end{align*}
Here, the term $T^{x_2}_{(23)(31),\{2\}}$ captures terms from the righthand side of~\eqref{eq:maj} satisfying $j < i = k-1$.

A key feature of this definition is that the product of constrained translates is a regular statistic.
Therefore, regular statistics are closed under multiplication.
Moreover, the size, shift and power of regular statistics are subadditive under products.
Note for $(U,V)$ a packed partial permutation and $|S| = |U \cup V|$ that the graphs of $(U,V)$ and $(S(U),S(V))$ are isomorphic.
Therefore, by applying Theorem~\ref{t:moment} and summing over terms in a constrained translate, we have:
\begin{theorem}
\label{t:regular-moment}
Let $\Psi$ be a regular statistic with size $k$, shift $q$ and power $p$, and let $d \geq 1$.
	Then there exists a polynomial $f_{\Psi,d} \in \ZZ[n,m_1,\dots,m_k]$ of degree at most $dp + dq$ where $\deg n = 1$ and $\deg m_i = i$ so that for any partition $\lambda$
	\[
	(n)_{dq}\EE_\lambda[\Psi^d] = f_{\Psi,d}\left(n,m_1(\lambda),m_2(\lambda),\dots,m_{dk}(\lambda)\right).
	\]
\end{theorem}

There are many previous results on moments of pattern counting statistics in the literature, beginning with Zeilberger's result that all moments for the random variable counting subwords of size $k$ in a uniformly random permutation with a given relative order is a polynomial in the size of the permutation~\cite{Zeilberger}.
Zeilberger's result has been generalized both to more general statistics~\cite{DK} and allowing for specific~\cite{KLY} or arbitrary cycle type~\cite{Hultman,GR,GP}.
Theorem~\ref{t:regular-moment} generalizes all previous results on this topic we are aware of.

We next apply Theorem~\ref{t:regular-moment} to understand the limiting behavior of regular statistics.
Let $\VV_\lambda$ denote the variance with respect to the uniform distribution on $K_\lambda$.
We characterize the asymptotic expectation and variance of a regular statistic:
\begin{theorem} \label{t:regular-asymptotics}
 {\em ($\subset$ Theorems~\ref{expectation-fixed-point} and~\ref{regular-variance})}
	Let $\Psi$ be a regular statistic with power $p$ and $\{\lambda^{(n)}\}$ be a sequence of partitions of $n$
	so that for $\alpha,\beta \in [0,1]$ we have
\begin{equation}
    \label{eq:limits}
		\lim_{n \to \infty} \frac{m_1(\lambda^{(n)})}{n} = \alpha
		\quad \mbox{and} \quad
		\lim_{n \to \infty} \frac{m_2(\lambda^{(n)})}{n} = \beta.
\end{equation}
Then there are polynomials $f,g,h \in \RR[x]$ so that for any partition $\lambda$
\begin{equation}\label{eq:intro-asymptotics}
	\lim_{n \to \infty} \EE_\lambda\left[ \frac{\Psi(X_n)}{n^p}\right] = f(\alpha), \quad \  
	\lim_{n \to \infty} \VV_\lambda\left[\frac{\Psi(X_n)}{n^{2p-1}}\right]= g(\alpha) + \beta h(\alpha).
\end{equation}
\end{theorem}

Note the scaling factors in~\eqref{eq:intro-asymptotics} can dominate the numerator, for instance when $\Psi = m_2$, the number of two cycles.
However, for many regular statistics in the literature this is the correct scaling factor and the limits are non-degenerate.
In this case, we use the variance result to prove a law of large numbers for regular statistics based on the limiting proportion of fixed points (see Corollary~\ref{regular-weak-law}).
Additionally, we can show random permutations drawn from a sequence of cycle types have a permuton limit based on the limiting proportion of fixed point (see Theorem~\ref{alpha-permuton}).

Since indicators $1_{IJ}$ and $1_{KL}$ are independent when $I\cup J$ and $K \cup L$ are disjoint, most terms in the na\"ive expansion of a regular statistic into $1_{IJ}$'s are independent random variables.
This suggests many regular statistics will be asymptotically normal. 
Hofer demonstrates this behavior for vincular pattern counts~\cite{Hofer}, generalizing B\'ona's result for classical pattern counts~\cite{Bona}.
When conditioning on arbitrary cycle type, the only statistics for which asymptotic normality was known prior to our work were the descent~\cite{KL1} and peak~\cite{FKL} statistics.
We will show how to extend Fulman's result for all long cycles from descents~\cite{Fulman} to all vincular statistics.
However, the state of the art has moved much further.
Using Theorem~\ref{t:regular-moment} and weighted dependency graphs, F\'eray and Kammoun have extended B\'ona's result and part of Hofer's to sequences cycle type of with limits as in~\eqref{eq:limits}~\cite{FK}.
See also~\cite{Dubach} for a less general asymptotic normality result independent of our work.

\subsection*{Organization} In Section~\ref{sec:Background}, we introduce key features of permutations and partial permutations.
We then prove Theorem~\ref{t:moment} in Section~\ref{sec:Moments}.
Next, in Section~\ref{sec:Regular} we prove Theorem~\ref{t:regular-moment} and discusses its relation to prior work.
Our asymptotic results including Theorem~\ref{t:regular-asymptotics} are in Section~\ref{sec:Asymptotic}.
We conclude with final remarks in Section~\ref{sec:Conclusion}.

\subsection*{Acknowledgements}
The authors are grateful to Michael Coopman, Valentin F\'eray, Yeor Hafouta, Mohammed Slim Kammoun, Gene Kim, Toby Johnson, Kevin Liu, 
Moxuan Liu,
 Arnaud Marsiglietti, 
James Pascoe, 
Bruce Sagan, John Stembridge,
and Yan Zhuang for helpful conversations.
We are especially grateful to Eric Ramos, who helped us conceptualize the project at its inception, and to the anonymous referee who dramatically simplified our proof of Theorem~\ref{t:moment}.
Z. Hamaker was partially supported by NSF Grant DMS-2054423.
B. Rhoades was partially supported by NSF Grant DMS-1953781 and DMS-2246846.

\section{Background}
\label{sec:Background}

Let $[n] = \{1,2,\dots,n\}$.
For $S$ a finite set, let $\binom{S}{k}$ be the set of $k$--element subsets of $S$.

\subsection{Partial Permutations}

We repeat several key definitions from the introduction.
For $S$ a finite set, let $\symm_S$ be the permutations of $S$.
In particular, $\symm_{[n]}$ is the usual symmetric group $\symm_n$.
A \emph{partial permutation} of size $k$ in $\symm_n$ is a bijection from $S$ to $T$ where $S, T \in \binom{[n]}{k}$.
We represent a partial permutation as a pair $(I,J)$ of tuples $I = (i_1,\dots,i_k) \in \symm_S$ and $J = (j_1,\dots,j_k) \in \symm_T$ where $i_1 \mapsto j_1,\dots, i_k \mapsto j_k$.
Let $\symm_{n,k}$ be the set of all partial permutations of size $k$ in $\symm_n$.
There are eighteen partial permutations in $\symm_{3,2}$; the six with $I = (12)$ are:
\begin{align*}
(12,12), (12, 13), (12,21), (12,23), (12,31), (12, 32).
\end{align*}
Note partial permutations $(I,J) \in \symm_{n,k}$ are unchanged when permuting the entries of $I$ and $J$ simultaneously, so we can always write $I$ in increasing order.

Each $(I,J) \in \symm_{n,k}$ with $I = (i_1, \dots, i_k)$ and $J = (j_1, \dots, j_k)$ has an associated  {\em graph} $G(I,J)$ whose vertex set is $I \cup J$ with edges $i_1 \rightarrow j_1, \dots, i_k \rightarrow j_k$.
This graph is an extension of the disjoint cycle notation for a permutation in $\symm_n$.
Every connected component of $G(I,J)$ is a directed path or a directed cycle.
The {\em cycle partition} $\mu$ and {\em path partition} $\nu$ of $(I,J)$ are the partitions whose parts are the cycle and path lengths in $G(I,J)$, respectively.
Here path lengths count edges, so $u_0 \to u_1 \to \dots \to u_\ell$ has length $\ell$.
The pair $(\mu,\nu)$ is the \emph{cycle-path type} of $(I,J) \in \symm_{n,k}$, which we denote $\mathrm{cyc}(I,J)$ and generalizes the cycle type of a permutation.
See Figure~\ref{fig:graph} for an example of these concepts.


\begin{remark}
    In our companion paper~\cite{HR2}, we take the convention that $G(I,J)$ has vertex set $n$, with vertices not in $I$ or $J$ isolated.
    Additionally, in that work we say a path's size is the number of vertices.
    These conventions are preferable for stating the symmetric function results in that paper, but would complicate the results in this paper.
\end{remark}

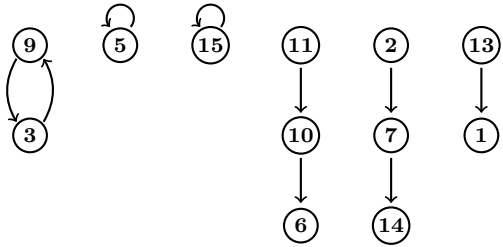
\begin{figure}
 \begin{center}
 \begin{tikzpicture}[scale = 0.6]
 
 \coordinate (v11) at (6,0);
 
 \coordinate (v10) at (6,-2);
 
 \coordinate (v6) at (6,-4);

 \coordinate (v2) at (8,0);
 
 \coordinate (v7) at (8,-2);
 
\coordinate (v14) at (8,-4);
  
 \coordinate (v13) at (10,0);
   
 \coordinate (v1) at (10,-2);
    
 \coordinate (v8) at (14,0);
 
 \coordinate (v4) at (12,0);

 \coordinate (v9) at (0,0);
 
 \coordinate (v3) at (0,-2);
 
 \coordinate (v5) at (2,0);
 
 \coordinate (v15) at (4,0);

   \node [draw, circle, fill = white, inner sep = 2pt, thick] at (v1)
  {\scriptsize {\bf 1} };
     \node [draw, circle, fill = white, inner sep = 2pt, thick] at (v2)
  {\scriptsize {\bf 2}};
     \node [draw, circle, fill = white, inner sep = 2pt, thick] at (v3)
  {\scriptsize {\bf 3}};
     \node [draw, circle, fill = white, inner sep = 2pt, thick] at (v5)
  {\scriptsize {\bf 5}};
     \node [draw, circle, fill = white, inner sep = 2pt, thick] at (v6)
  {\scriptsize {\bf 6}};
     \node [draw, circle, fill = white, inner sep = 2pt, thick] at (v7)
  {\scriptsize {\bf 7}};
     \node [draw, circle, fill = white, inner sep = 2pt, thick] at (v9)
  {\scriptsize {\bf 9}};
     \node [draw, circle, fill = white, inner sep = 1pt, thick] at (v10)
  {\scriptsize {\bf 10}};
     \node [draw, circle, fill = white, inner sep = 1pt, thick] at (v11)
  {\scriptsize {\bf 11}};
     \node [draw, circle, fill = white, inner sep = 1pt, thick] at (v13)
  {\scriptsize {\bf 13}};
     \node [draw, circle, fill = white, inner sep = 1pt, thick] at (v14)
  {\scriptsize {\bf 14}};
     \node [draw, circle, fill = white, inner sep = 1pt, thick] at (v15)
  {\scriptsize {\bf 15}};
  
  \draw [->, thick] (6,-0.5) -- (6,-1.5);
  \draw [->, thick] (6,-2.5) -- (6,-3.5);
 \draw [->, thick] (8,-0.5) -- (8,-1.5);
 \draw [->, thick] (8,-2.5) -- (8,-3.5);
 \draw [->, thick] (10,-0.5) -- (10,-1.5);
 \draw [->, thick]  (-0.3,-0.3) to[bend right]  (-0.3,-1.7);
 \draw [->, thick]  (0.3,-1.7) to[bend right]  (0.3,-0.3);
 
 \draw[->, thick]  ($(2,0.6) + (-40:3mm)$) arc (-40:220:3mm);

 \draw[->, thick]  ($(4,0.6) + (-40:3mm)$) arc (-40:220:3mm);
 
 \end{tikzpicture} 
 \end{center}
 \caption{For $I = (2,3,5,7,9,10,11,13,15), J = (7,9,5,14,3,6,10,1,15) \in \symm_{15,9}$ we depict the graph $G(I,J)$.
 The cycle and path partition are $(2,1,1)$ and $(2,2,1)$.}
 \label{fig:graph}
 \end{figure}

The symmetric group $\symm_n$ acts naturally on length $k$ lists of elements of $[n]$ by the rule 
\begin{equation}
w(I) = w((i_1, \dots, i_k)) := (w(i_1), \dots, w(i_k))
\end{equation}
where $I = (i_1, \dots, i_k)$.
This induces an action of $\symm_n$ on $\symm_{n,k}$ by the diagonal rule 
\begin{equation}
w(I,J) := (w(I), w(J)).
\end{equation}
The orbits of this action are parametrized by cycle-path type.
In the introduction, we define $1_{IJ}$ by
\[
1_{IJ}(w) = \begin{cases}
    1 & (i_\ell) = j_\ell \text{ for } \ell \in [k],\\
    0 & \text{else}
\end{cases}
\]
for $w \in \symm_n$.
We remark that for $u$ also in $\symm_n$ that
\begin{equation}
    \label{eq:1IJ-prod}
    1_{IJ}(u^{-1}wu) = 1_{u(I)u(J)}(w).
\end{equation}

\begin{proposition}
\label{p:cycle-path}
	Partial permutations $(I,J), (I',J') \in \symm_{n,k}$
	have the same cycle-path type if and only if there is a permutation $w \in \symm_n$ so that $I' = w(I)$ and $J' = w(J)$.
\end{proposition}

\begin{proof}
For $w \in \symm_n$ so that $I' = w(I)$ and $J' = w(J)$, note that $w$ restricts to a graph isomorphism from $G(I,J)$ to $G(I,J')$.
Since $(I,J)$ and $(I',J')$ have the same cycle-path type if and only if $G(I,J)$ and $G(I',J')$ are isomorphic as graphs, the converse follows.
For the forwards direction, take a graph isomorphism and extend it to a permutation in $\symm_n$.
\end{proof}

In the statement of Proposition~\ref{p:cycle-path}, note that the partial permutations $(I,J)$ and $(I',J')$ of $[n]$ do not a priori have the same size.
In particular, we see that
cycle-path type determines size.
In fact, for $(I,J)$ a partial permutation of size $k$ with cycle-path type $(\nu,\mu)$ we have $k = |\mu| + |\nu|$.

\subsection{Probability and Permutation Statistics}
We review some basics of probability, especially as they apply to permutation statistics.
For $\Omega$ a finite set, a \emph{probability measure} $\mu$ is a function $\mu:\Omega \to [0,1]$ so that $\sum_{s \in \Omega} \mu(s) = 1$.
An \emph{event} $A \subseteq \Omega$ has \emph{probability} $\PP[A] = \sum_{a \in A} \mu(a)$.
A \emph{random variable} is a function $X:\Omega \to \mathbb{R}$, and its \emph{expected value} is
\[
\EE[X] = \sum_{s \in \Omega} X(s) \mu(s).
\]
Similarly, the \emph{variance} of $X$ is $\VV[X] = \EE[X^2] - \EE[X]^2$ and the \emph{covariance} of two random variables $X$ and $Y$ is $\mathrm{Cov}[X,Y] = \EE[XY]-\EE[X]\EE[Y]$.

In this paper, we are primarily interested probabilistic aspects of permutation statistics, which are functions $f: \sqcup_{n \geq 1} \symm_n \to \mathbb{R}$.
Most of our results are for the uniform measure on $K_\lambda$, which is the set of permutations with cycle type $\lambda$ (or cycle-path type $(\lambda,\varnothing)$).
Here, the expected value and variance are
\[
\EE_\lambda[f] = \frac{1}{|K_\lambda|}\sum_{w \in K_\lambda} f(w) \quad \mbox{and} \quad \VV_\lambda[f] =  \EE_\lambda[f^2 - \EE_\lambda[f]^2].
\]
Alternatively, for $v \in K_\lambda$, we have
\[
\EE_\lambda[f] = \frac{1}{n!} \sum_{w \in \symm_n} f(w^{-1} v w).
\]
This alternate formula proves incredibly valuable.
Define $R_n\,f:\symm_n \to \RR$ by
\begin{equation}
    \label{eq:reynolds}
R_n\,f(v) = \frac{1}{n!} \sum_{w \in \symm_n} f(w^{-1} v w) = \EE_{\mathrm{cyc}(v)}[f],
\end{equation}
where $\mathrm{cyc}(v)$ is the cycle type of $v$.
Here $R_n$ is the \emph{Reynolds operator} (or {\em averaging operator}) which plays a prominent role in  representation theory.
Equation~\eqref{eq:reynolds} shows that the function $R_n\,f: \symm_n \to \CC$ encodes all of the expectations $\EE_\lambda[f]$ simultaneously.


\section{Expectation for Partial Permutations by Cycle Type}
\label{sec:Moments}

The purpose of this section is to prove Theorem~\ref{t:moment}.
For the remainder of the section, fix a partial permutation $(I,J)$ of size $k$ with $|I\cup J| = m$ and cycle--path type $(\mu,\nu)$.
Further, we can assume $(I,J)$ is packed, so $I\cup J = [m]$.
We must construct a polynomial $f_{(\mu,\nu)}(n,m_1,\dots,m_k)$ so that for each $\lambda \vdash n$
\[
(n)_m \cdot \EE_{\lambda}[1_{IJ}] = f_{(\mu,\nu)}(n,m_1(\lambda),m_2(\lambda),\dots,m_k(\lambda)
\]
(this is~\eqref{eq:moment}).
We first reinterpret~\eqref{eq:moment} combinatorially.
If $X$ is a uniformly random member of $K_\lambda$, $\pi \in K_\lambda$ and $Y$ a uniformly random permutation,
\[
\EE_\lambda[1_{IJ}] = \EE [1_{IJ}(X)] = \mathbb{P}[1_{IJ}(X) = 1] = \mathbb{P}[1_{IJ}(Y \pi Y^{-1}) = 1] = \mathbb{P}[1_{Y(I)Y(J)}(\pi) = 1].
\]
In particular, we can compute this probability using the restriction of $Y$ to $I \cup J$, which we view as an injection.
Towards this end, for the remainder of this section fix  $\pi \in K_\lambda$.
We say a function $\psi:[m] \to [n]$ (not necessarily injective) is \emph{compatible} with $(I,J)$ relative to $\pi \in \symm_n$ if
\begin{equation}
\label{eq:condition}
	\quad \pi(\psi(i_k)) = \psi(j_k)
\end{equation}
for all $k \in [m]$.
Then
\begin{equation}
(n)_m \EE_\lambda [1_{IJ}] = \#\{ \text{injections } \phi:[m] \hookrightarrow [n] \mid \phi\ \text{is compatible with}\ (I,J)\ \text{relative to} \ \pi \}.	
\label{eq:reformulation}
\end{equation}

We analyze \eqref{eq:reformulation} using M\"obius inversion, and to this end we define an auxiliary set of functions. For a fixed packed partial permutation $(I,J) \in \symm_{k,n}$ with $|I \cup J| = m$ and a fixed $\pi \in K_\lambda$, define
\begin{equation}
    \mathcal{C}_{IJ}(\lambda) := \{ \text{functions } \psi:[m] \to[n] \mid \psi \text{ is compatible with $(I,J)$ relative to $\pi$} \}.
\end{equation}
Functions  $\psi \in \mathcal{C}_{IJ}(\lambda)$ are not necessarily injections.
 Observe that the cardinality of $\mathcal{C}_{IJ}(\lambda)$ does not depend on the choice of $\pi \in K_\lambda$. We also have $|C_{IJ}(\lambda)| = |C_{I'J'}(\lambda)|$ whenever $(I,J)$ and $(I',J')$ have the same cycle-path types. We may therefore define
 $c_{(\mu,\nu)}(\lambda) := |\mathcal{C}_{IJ}(\lambda)|$ where $(I,J)$ has cycle-path type $(\mu,\nu)$.

\begin{proposition}
	\label{p:function-count}
	For $\lambda,\mu,\nu$ partitions with $|\mu| + |\nu| = k$ and $|\lambda| = n$,
	\[
	c_{(\mu,\nu)}(\lambda ) \in \ZZ[n,m_1(\lambda),\dots,m_n(\lambda)]
	\]
has degree $k$ where $\deg n = 1$ and $\deg m_i = i$.
\end{proposition}

\begin{proof}
We establish an explicit formula for $c_{(\mu,\nu)}(\lambda)$ which satisfies the required properties.
For $\psi \in \mathcal{C}_{IJ}(\lambda)$, we see each $\ell$-cycle  $(s_1,\dots,s_\ell,s_1)$ in $(I,J)$ must be mapped under $\psi$ to an $\ell$-cycle of $\pi$ and each length $\ell$ path $(p_1,\dots,p_{\ell+1})$ in $(I,J)$ must be mapped under $\psi$ to part of a cycle of $\pi$ with length at least $\ell+1$.
For $k$ in a such a cycle, there is one way to map $p_1$ to $k$, after which $p_2,\dots,p_{\ell+1}$ are determined.
Therefore, the number of ways to map the path $(p_1, \dots, p_{\ell+1})$ into $[n]$ is 
\begin{equation}
\label{eq:path}
	n - \sum_{i=1}^{\ell} i \cdot m_i(\lambda).
\end{equation}
For $(I,J)$ of cycle--path type $(\mu,\nu)$, we then see
\begin{equation}
\label{eq:good-count}
	c_{(\mu,\nu)}(\lambda) = \prod_{i=1}^m (i\cdot m_i(\lambda))^{m_i(\mu)} \cdot \prod_{\ell=1}^m \left(n - \sum_{i=1}^\ell i \cdot m_i(\lambda) \right)^{m_\ell(\nu)}.
\end{equation} 
If we let $\deg n = 1$ and $\deg m_k(\lambda) = k$, the first product has total degree $|\mu|$.
The degree of~\eqref{eq:path} is $\ell$, so the total degree of the second term is $|\nu|$, hence~\eqref{eq:good-count} is a polynomial in the variables $n,m_1(\lambda),m_2(\lambda),\dots,m_m(\lambda)$ whose degree is at most $|\mu| + |\nu|$.
\end{proof}

Now that we have a count for all such functions, we use M\"obius inversion on the set partition lattice to find the number of injections from $[m]$ to $[n]$ that are compatible with $(I,J)$ relative to $\pi$.
Let $(\Pi_m,\leq)$ denote the poset of set partitions of $[m]$ ordered by refinement. For $\rho,\tau \in \Pi_n$ we have 
\[
\rho = \{\rho_1,\dots,\rho_k\} \leq \tau = \{\tau_1,\dots,\tau_m\}
\]
 if for each $i \in [k]$ there is a $j \in [m]$ so that $\rho_i \subseteq \tau_j$.
 The poset $\Pi_m$ is a lattice with minimum element $\hat{0}_m := \{ \{1\},\{2\},\dots,\{m\}\}$.


For $\rho \in \Pi_m$, define
\[
f_{(\mu,\nu)}(\rho) = \{\psi \in \mathcal{C}_{IJ}(\lambda): \psi(i) = \psi(j) \ \mbox{if}\  i,j\ \mbox{are in the same block} \}
\]
and
\[
g_{(\mu,\nu)}(\rho) = \{\psi \in \mathcal{C}_{IJ}(\lambda): \psi(i) = \psi(j) \ \mbox{if and only if}\  i,j\ \mbox{are in the same block} \},
\]
so $f_{(\mu,\nu)}(\tau) = \sum_{\rho \leq \tau} g_{(\mu,\nu)}(\rho)$.
By definition, $f_{(\mu,\nu)}(1_m) = c_{(\mu,\nu)}(\lambda)$.
We now characterize the values $f_{(\mu,\nu)}(\rho)$ can take.

\begin{proposition}
\label{p:poset-to-partition}
For $\rho \in \Pi_m$, either $f_{(\mu,\nu)}(\rho) = 0$ or  there exist $\mu'$ and $\nu'$ with $m_i(\mu') \leq m_i(\mu)$ for all $i$ and $m_1(\nu') \leq m_1(\nu)$ so that
\[
f_{(\mu,\nu)}(\rho) = c_{(\mu',\nu')}(\lambda).
\]

\end{proposition}

\begin{proof}
When $\rho = \hat{0}_m$, the result holds with $(\mu',\nu') = (\mu,\nu)$.
When $\rho$ contains a non-singleton block, we show either $f_{(\mu,\nu)}(\rho)$ vanishes or construct a partial permutation $(\tilde{I},\tilde{J})$ of cycle--path type $(\tilde{\mu},\tilde{\nu})$ and set partition $\tilde{\rho}$ of a smaller set so that $f_{(\mu,\nu)}(\rho) = f_{(\tilde{\mu},\tilde{\nu})}(\tilde{\rho})$ and $m_1(\tilde{\nu}) \leq m_1(\nu)$.
Iterating this construction if necessary, the result will follow.

Since $\rho \neq \hat{0}_m$, we can pick $i,j$ in the same block in $\rho$.
If $i$ and $j$ are in the same path or cycle, we see any $\psi \in \CCC_{IJ}(\lambda)$ maps $i$ and $j$ to different values so $f_\lambda(\rho)$ must be $0$. We may therefore assume that $i$ and $j$ belong to distinct paths or cycles. Our analysis breaks into cases depending on whether either or both of $i$ and $j$ belong to cycles.

{\bf Case 1:} {\em $i$ and $j$ belong to distinct cycles 
\[
i = c_1 \to \dots \to c_s \to c_1 = i,\quad j = d_1 \to \dots \to d_r \to d_1 = j.
\]
of lengths $s$ and $r$, respectively}

In order for $\psi \in C_{IJ}(\lambda)$ to satisfy $\psi(j) = \psi(i)$ we must have $\psi(c_\ell) = \psi(d_\ell)$ for all $\ell \in [s]$.
In particular, we have $f_{(\mu,\nu)}(\lambda) = 0$ if $r \neq s$, so we assume that $r = s$. 
Next replace values $d_1,\dots,d_r$ in $(I,J)$ with $c_1,\dots,c_r$ to obtain $(\tilde{I},\tilde{J})$.
Let $\tilde{\rho}$ be the set partition of $\tilde{I} \cup \tilde{J}$ obtained as follows.
\begin{itemize}
    \item Let $\tau \in \Pi_m$ be the set partition whose only non-singleton blocks are $\{c_1,d_1\}, \{c_2,d_2\}, \dots, \{c_r,d_r\}.$
    \item Construct the join $\tau \vee \rho$ in the set partition lattice $\Pi_m$.
    \item Delete each occurrence of $d_i$ from $\tau \vee \rho$ to obtain $\tilde{\rho}$.
\end{itemize}
The set partition $\tilde{\rho}$ encodes all the equivalences induced by insisting $\phi(i) = \phi(j)$, so we have $f_{(\mu,\nu)}(\rho) = f_{(\tilde{\mu},\tilde{\nu})}(\tilde{\rho})$ as desired.
Note that $\tilde{\nu} = \nu$ in this case.

{\bf Case 2:} {\em 
$i$ belongs to an $s$--cycle as in Case 1 and $j$ belong to a path
\[
d_0 \to d_1 \to \dots \to j = d_p \to \dots \to d_r.
\]}

We must have $\psi(d_\ell) = \psi(c_{\ell -p})$ where $\ell - p$ is interpreted mod $s$.
In particular, if $r \geq s$ we have $f_{(\mu,\nu)}(\lambda) = 0.$ 
When $r < s$, remove $d_0,d_1,\dots,d_r$ from $(I,J)$ to construct $(\tilde{I},\tilde{J})$, then replace all instances of $d_\ell$ with $c_{\ell-p}$ in $\rho$.
Merging blocks as in Case 1, we obtain a set partition $\tilde{\rho}$ with the desired properties.
Here, we see $\tilde{\mu} = \mu$ and $\tilde{\nu}$ is obtained from $\nu$ by removing a single $r$.

{\bf Case 3:} {\em $i$ and $j$ belong to distinct paths
\[
c_0 \to c_1  \to \dots \to i = c_q \to \dots \to c_s, \quad d_0 \to d_1 \to \dots \to j = d_p \to \dots \to d_r.
\]
with $p \leq q$.} 

We must have $\psi(d_\ell) = \psi(c_{\ell+q-p})$ whenever both quantities are defined.
Here, to construct $(\tilde{I},\tilde{J})$ we merge the two paths by removing the values $d_i$ for which $i + q - p \geq 0$, now mapping $c_s$ to $d_{s-q+p+1}$ if such a term exists.
No $d$ can map to $c_0$ since  $p \leq q$.
Note the resulting path has length at least $\max\{r,s\}$, with equality only when one path is subsumed by the other.
In particular, for $(\tilde{\mu},\tilde{\nu})$ the cycle--path type of $(\tilde{I},\tilde{J})$ we have $m_1(\tilde{\nu}) \leq m_1(\nu)$.
Again, we merge values in $\rho$ to produce $\tilde{\rho}$ as above and see  that $f_{(\mu,\nu)}(\rho) = f_{(\tilde{\mu},\tilde{\nu})}(\tilde{\rho})$ as desired.

This completes our proof.
\end{proof}

An example may clarify the proof of Proposition~\ref{p:poset-to-partition}. Let $m = 13$ and let $(I,J)$ be the partial permutation whose directed graph is given by
\[ 
1 \to 2 \to 3 \to 4 \to 5 \quad \quad 6 \to 7 \to 8 \to 9 \quad \quad 10 \to 11 \quad \quad 12 \to 13.
\]
We let $\rho \in \Pi_m$ be the set partition
\[ \rho := \{ \{1,12\}, \, \{2 \}, \, \{3,11\}, \, \{4,7,10\}, \, \{5\}, \, \{6,13\}, \, \{8\}, \, \{9\} \}. \]
Observe that $i = 4$ and $j = 7$ are in the same block of $\rho$. Since $i$ and $j$ belong to distinct paths of $(I,J)$, we are in Case 3 of the proof of Proposition~\ref{p:poset-to-partition}. If $\psi: [m] \to [n]$ is counted by $f_{(\mu,\nu)}(\rho)$ we have $\psi(4) = \psi(7)$, which forces $\psi(3) = \psi(6)$ and $\psi(5) = \psi(8)$. If we let $\tau \in \Pi_m$ be the set partition
\[ \tau := \{ \{1\}, \, \{2\}, \, \{3,6\}, \, \{4,7\}, \, \{5,8\}, \, \{9\}, \, \{10\}, \, \{11\}, \, \{12\}, \, \{13\} \}\]
encoding the equalities $\psi(4) = \psi(7)$, $\psi(3) = \psi(6)$ and $\psi(5) = \psi(8)$, we compute the join
\[ \tau \vee \rho := \{ \{1, 12 \}, \, \{2\}, \, \{3, 6, 11, 13\}, \{ 4,7,10\}, \, \{5,8\}, \{9\}\}.\]
We obtain $(\tilde{I},\tilde{J})$ by merging paths to obtain 
\[
1 \to 2 \to 3 \to 4 \to 5 \to 9 \quad \quad 10 \to 11 \quad \quad 12 \to 13.
\]
Observe that the elements $6,7,8$ appearing in $(I,J)$ which are missing from $(\tilde{I},\tilde{J})$ appear in distinct 2-element blocks of $\tau$. In order to obtain $\tilde{\rho}$, we erase the elements $6,7,8$ from $\tau \vee \rho$; this yields
\[ \tilde{\rho} = \{ \{1,12\}, \, \{2\}, \, \{3,11, 13\}, \, \{4,10\}, \, \{5\}, \, \{9\} \}.\]
If $(\tilde{\mu},\tilde{\nu})$ is the cycle-path type of $(\tilde{I},\tilde{J})$, we have $f_{(\mu,\nu)}(\rho) = f_{(\tilde{\mu},\tilde{\nu})}(\tilde{\rho})$.

Let $\mu: \Pi_m \times \Pi_m \to \mathbb{Z}$ be the M\"obius function of the set partition lattice $\Pi_m$. We  compute $g(\hat{0}_m)$ using M\"obius inversion to prove our Theorem~\ref{t:moment}.

\begin{proof}[Proof of Theorem~\ref{t:moment}]
For a set partition $\rho = \{\rho_1,\dots,\rho_k\} \in \Pi_m$ it is well-known that the lower M\"obius invariant $\mu(\hat{0}_m,\rho)$ is given by the formula
\[
\mu(\hat{0}_m,\rho) = (-1)^{|\rho| - m} \prod_{i=1}^k (|\rho_i|-1)!.
\]
Then by M\"obius inversion
\begin{equation}
\label{eq:g}
g(\hat{0}_m) = \sum_{\rho \in \Pi_n} \mu(\hat{0}_m,\rho) \cdot f(\rho).
\end{equation}
By combining Propositions~\ref{p:function-count} and~\ref{p:poset-to-partition}, we know $f(\rho)$ is a polynomial of degree at most $k$.
The other terms in~\eqref{eq:g} are functions of $m$, independent of $n$.
Therefore, the result follows.
\end{proof}

Later, we will require a deeper understanding of maximal degree terms in $f_{(\mu,\nu)}$ as defined in Theorem~\ref{t:moment}.

\begin{corollary}
    \label{c:top-degree-moment}
Let $(I,J)$ be a partial permutation of size $k$ with cycle--path type $(\mu,\nu)$.
Then each degree $k$ monomials of $f_{(\mu,\nu)}$ is divisible by $n^{a_0}\cdot m_1^{a_1}$ for some $a_0,a_1$ with $a_0 + a_1 \leq m_1(\nu)$ and $a_1 \geq m_1(\mu)$.
\end{corollary}

\begin{proof}
In~\eqref{eq:good-count} we see the contribution to maximal degree terms must include all terms corresponding to $\mu$, so a maximal degree monomial will be divisible by $m_1^{m_1(\mu)}$.
Additionally, we see terms corresponding to $\nu$ can only feature $n$ when $\ell = 1$, so we have a term $(n - m_1)^{m_1(\nu)}$.
From Proposition~\ref{p:poset-to-partition}, we see for all terms in~\eqref{eq:g} that $f(\rho) = c_{(\mu',\nu')}$ of degree $k$ that $\mu' = \mu$ and $m_1(\nu') \leq m_1(\nu)$.
The result now follows.
\end{proof}
	
This method gives some degree of control when $m$ diverges much more slowly than $n$.
The number of terms in~\eqref{eq:g} is the $m$th Bell number $B_m < \left(\frac{.7962 m}{\log(m+1)} \right)^m$.
The M\"obius function can be bounded by
\[\mu(\rho,1_m) \leq \sqrt{2\pi m} \left(\frac{m}{e}\right)^m e^{m/12},
\]
so a coarse upper bound for the regime where this argument applies is when
\[
\sqrt{2\pi m} \left(\frac{m}{e}\right)^m e^{m/12} \cdot \left(\frac{.7962 m}{\log(m+1)} \right)^m \leq n^k
\]
for some fixed $k$.
To a first order approximation, this shows the control of degree extends to the regime where $m = O(\log n/\log \log n)$.
Further improvements should be possible by estimating the portion of set partitions $\rho$ for which $f(\rho)$ does not vanish, which depends on $\lambda$ in a non-trivial way.

\section{Regular Statistics}
\label{sec:Regular}

From the introduction, recall for $(U,V)$ a partial permutation with $U \cup V = [m]$, $C \subseteq [m-1]$ and $f \in \CC[x_1,\dots,x_m]$ that the associated \emph{constrained translate}
\[
T^f_{(U,V),C} = \sum_{L = \{\ell_1 < \dots < \ell_m\} \in \binom{\NN}{m}_C} f(\ell_1,\dots,\ell_m) \cdot 1_{L(U)\,L(V)}
\]
has \emph{size} $k$, \emph{shift} $|C|$ and \emph{power} $k + \deg f - |C|$.
We call $\{(U,V),C,f\}$ a \emph{packed triple}.
Regular statistics are linear combinations of constrained translates.
In this section, first we prove Theorem~\ref{t:regular-moment} on moments of regular statistics.
Then we show that many classical permutation statistics  are regular statistics, allowing us to recover a litany of prior results as special cases of Theorem~\ref{t:regular-moment}.

\subsection{Proof of Theorem~\ref{t:regular-moment}}
\label{ss:regular-moment}

The purpose of this subsection is to prove Theorem~\ref{t:regular-moment}, which states that moments of regular statistics applied to a uniformly random element of $K_\lambda$ are polynomials satisfying certain degree conditions.
Since regular statistics are linear combinations of constrained translates, the expected value result will follow by linearity of expectation for constrained translates.
To extend the result to higher moments, we show products of regular statistics are also regular and that their degrees are subadditive.
This last feature is a key advantage of regular statistics relative to the bivincular statistics considered in~\cite{DK}, streamlining their arguments even in the case where $X$ is uniform on $\symm_n$.

To prove the moment result holds for constrained translates, we need a preliminary results to deal with the polynomial weight $f$ in the constrained translate $T^f_{(U,V),C}$.

\begin{lemma}
\label{l:vincular-poly}
	Let $f \in \CC[x_1,\dots,x_k]$ and $C \subseteq [k-1]$ with $|C| = q$.
	Then there exists $\overline{f}_C \in \mathbb{C}[n]$ with $\deg \overline{f}_C= \deg f$ so that
	\begin{equation}
	\label{eq:vincular-poly}
	\sum_{\{i_1< \dots < i_k\} \in \binom{[n]}{k}_C} f(i_1,\dots,i_k) = \overline{f}_C(n) \binom{n-q}{k-q}.
	\end{equation}
\end{lemma}

\begin{proof}
We first prove the case where $C = \varnothing$.
By linearity we may assume that $f = x_1^{a_1}  \cdots x_m^{a_m}$ is a monomial in $x_1, \dots, x_m$. 
The result is clear for $m = 0$. Given $m > 0$
the function
\[
F(n) := \sum_{\{i_1<\dots<i_m\} \in \binom{[n]}{m}} f(i_1,\dots,i_m) = 
\sum_{\{i_1<\dots<i_m\} \in \binom{[n]}{m}} (i_1)^{a_1} \cdots (i_m)^{a_m}
\]
satisfies 
\begin{equation}
\label{F-recursion}
F(n) - F(n-1) = n^{a_m} \cdot \sum_{\{i_1<\dots<i_{m-1}\} \in \binom{[n-1]}{m-1}} (i_1)^{a_1} \cdots (i_{m-1})^{a_{m-1}}
\end{equation}
where the summation $n^{a_m} \cdot \sum_{\{i_1<\dots<i_{m-1}\} \in \binom{[n-1]}{m-1}} (i_1)^{a_1} \cdots (i_{m-1})^{a_{m-1}}$ is inductively a polynomial
in $n$ of degree $\deg f + m - 1$.  Since Equation~\eqref{F-recursion}
 holds for all $n > 0$, we see that $F(n)$ is a polynomial in $n$ of degree $\deg f + m$.
 It is evident that $F(i) = 0$ for $i = 0, 1, \dots, m-1$ so $F(n)$ is divisible by $\binom{n}{m}$ in the ring $\CC[n]$.

For $j \in [k]$, let $c_j = |\{c \in C: c < j\}|$.
The shift map $\mathrm{sh}_C:\binom{[n]}{k}_C \to \binom{[n-p]}{k-p}$ defined by
\[
\mathrm{sh}_C(\{i_1< \dots < i_k\}) = \{i_j - c_j: j-1 \notin C\}
\]
is a bijection.
Viewing $f$ as a function of $\mathrm{sh}_C(I)$, the result follows from the case where $C = \varnothing$.
\end{proof}

We can now prove the analogue of Theorem~\ref{t:regular-moment} for constrained translates.

\begin{proposition}
    \label{p:translate-moment}

Let $(U,V)$ be a size $k$ packed partial permutation with cycle--path type $(\mu,\nu)$, $U \cup V = [m]$, $f \in \CC[x_1,\dots,x_k]$ and $C \subseteq [m-1]$ with $|C| = q$ so the power of $T^f_{(U,V),C}$ is $p = k + \deg f - q$.
Then for $\lambda \vdash n$ we have
\[
(n)_m \EE_\lambda \left[T^f_{(U,V),C}\right] = \binom{n-q}{k-q}\overline{f}_C(n)f_{(\mu,\nu)}(n,m_1(\lambda),\dots,m_k(\lambda))
\]
where $\overline{f}_C$ and $f_{(\mu,\nu)}$ are as defined in Lemma~\ref{l:vincular-poly} and Theorem~\ref{t:moment}, respectively.
In addition, we have $\deg (n)_m \EE[T^f_{(U,V),C}] = p + k$ where $\deg n = 1$ and $\deg m_i = i$.
\end{proposition}

\begin{proof}
    Since the cycle--path type of every term in $T^f_{(U,V),C}$ is $(\mu,\nu)$, we have
\begin{align*}
     \EE[T^f_{(U,V),C}(X)] &= \sum_{L = \{\ell_1 < \dots < \ell_m\} \in \binom{\NN}{m}_C} f(\ell_1,\dots,\ell_m) \cdot \EE[1_{L(U)\,L(V)}(X)]\\
& = \sum_{L = \{\ell_1 < \dots < \ell_m\} \in \binom{\NN}{m}_C} f(\ell_1,\dots,\ell_m) \cdot \frac{1}{(n)_m}f_{(\mu,\nu)}(n,m_1(\lambda),\dots,m_k(\lambda))\\
& = \binom{n-q}{k-q}\frac{1}{(n)_m} \overline{f}_C(n)f_{(\mu,\nu)}(n,m_1(\lambda),\dots,m_k(\lambda))
\end{align*}
with the first equality by linearity of expectation, the second by Theorem~\ref{t:moment} and the third by Lemma~\ref{l:vincular-poly}.
The degree follows since $\deg \overline{f}_C = \deg f$, $\deg f_{(\mu,\nu)} = k$ and $\deg \binom{n-q}{k-q} = k-q$.
\end{proof}

By linearity of expectation, Proposition~\ref{p:translate-moment} implies Theorem~\ref{t:regular-moment} for the first moment ($d=1$).
To extend this to higher moments, we will show the product of regular statistics is regular and that the size, shift and power are subadditive.
Before showing this, we must define these quantities precisely, which requires a degree of care since constrained translates are not linearly independent.

\begin{definition}
    \label{def:regular-size}
For $\Psi$ a regular statistic with
\[
\Psi = \sum_{I = \{(U,V),C,f)\} \in \III} a_I T^f_{(U,V),C},
\]
the \emph{size}, \emph{shift} and \emph{power} of $\Psi$ for this expansion are, respectively, the maxima of $|U|,|C|$ and $|U|+\deg f - |C|$ among all $I$.
The \emph{size}, \emph{shift} and \emph{power} of $\Psi$ are the minimal size, shift and power among all expressions of $\Psi$ as a linear combination of constrained translates.
\end{definition}

\begin{proposition}
    \label{p:regular-algebra}
Let $\Phi,\Psi$ be regular statistics with sizes $k_1,k_2$, shifts $q_1,q_2$ and powers $p_1,p_2$, respectively.
Then $\Phi \cdot \Psi$ is a regular statistic with size at most $k_1 + k_2$, shift at most $q_1 + q_2$ and power at most $p_1 + p_2$.
\end{proposition}

\begin{proof}
Let $((U,V),C,f)$ and $((X,Y),D,g)$ be packed triples with $U \cup V = [m], X \cup Y = [\ell]$.
Applying linearity, it is enough to show this result when  $\Phi = T^f_{(U,V),C}$ and $\Psi =  T^g_{(X,Y),D}$.

We  expand the product $\Phi \cdot \Psi = T^f_{(U,V),C} \cdot T^g_{(X,Y),D}$ as a linear combination of indicator functions $\one_{I,J}$.
A typical term in this expansion  looks like
\begin{equation}
\label{fg-product}
(f(M) \cdot \one_{M(U),M(V)}) \cdot (g(L) \cdot \one_{L(X),L(Y)}) = f(M)g(L)\cdot \one_{M(U),M(V)} \cdot \one_{L(X),L(Y)}
\end{equation}
for a pair of sets $M \in \binom{[n]}{m}_C$ and $L \in \binom{[n]}{\ell}_D$.
Here we adopt the shorthand $f(M) := f(a_1, \dots, a_m)$ for $M = \{a_1 < \cdots < a_m \}$ and similarly for $g(L)$.

If the expression \eqref{fg-product} does not vanish 
it must equal $f(M)g(L) \cdot \one_{I,J}$ for some partial permutation $(I,J)$ with $I \cup J = M \cup L := S$ and $|S| = r \leq \ell+m$.
We write $(I,J) = (S(U'),S(V'))$ for a unique packed partial permutation $(U',V')$ with $U'$ increasing and set $S$.
The vincular constraints given by $C$ and $D$ on $S$ are given by the set $I(C) \cup J(D) \subseteq S$. 
The compression  $C' = [r](I(C)\cup J(D))$ of the set $C$ gives equivalent vincular constraints on the packed partial permutation
 $(U',V')$.
Furthermore, the product
 $f(M) \cdot g(L) = h(M\cup L)$ is a polynomial with $\deg h = \deg f + \deg g$ 
 so $T^{h}_{(U,'V'),C'}$ occurs as a summand in $T^f_{(U,V),C} \cdot T^g_{(X,Y),D}$.
Repeated subtraction shows that $T^f_{(U,V),C} \cdot T^g_{(X,Y),D}$ is regular.
The largest possible value of the size $|U'| = |V'|$ is $k_1 + k_2$; this occurs when $M \cap L = \varnothing$.
Similarly, the power $|U'| + \deg h - |C'| = |V'| + \deg h - |C'|$ is $p_1 + p_2$; this also occurs when $M \cap L = \varnothing$.
Finally, the shift achieves maximum value $|I(C) \cup J(D)| = q_1+q_2$ when $I(C) \cap J(D) = \varnothing$.
While these values are all attained for a product of constrained translates, cancellation can occur, hence the inequality.
\end{proof}

An important consequence of Proposition~\ref{p:regular-algebra} is the fact that regular statistics form an algebra.
We are now prepared to prove Theorem~\ref{t:regular-moment}.

\begin{proof}[Proof of Theorem~\ref{t:regular-moment}]
Since $\Psi$ is a regular statistic with size $k$, shift $q$ and power $p$, by Proposition~\ref{p:regular-algebra} we see $\Psi^d$ us a regular statistic with size at most $dk$, shift at most $dq$ and power at most $dp$.
The result now follows by Proposition~\ref{p:translate-moment} and linearity of expectation.
\end{proof}

By replacing Theorem~\ref{t:moment} with the observation for $(I,J)$ of size $k$ that
\[
\EE_{\symm_n}[1_{IJ}] = \frac{1}{(n-k)!},
\]
we obtain an analogue of Theorem~\ref{t:regular-moment}.

\begin{theorem}
\label{t:uniform-regular-moment}
    Let $\Psi$ be a regular statistic with size $k$, shift $q$ and power $p$.
    Then there exists a polynomial $e_{\Psi,d} \in \ZZ[n]$ of degree at most $dp + dq$ so that
\[
(n)_{dq} \EE_{\symm_n}[\Psi^d] = e_{\Psi,d}(n).
\]
\end{theorem}

\begin{proof}
By Proposition~\ref{p:regular-algebra} and linearity of expectation, the result will follows by an analogue of Proposition~\ref{p:translate-moment} for the uniform distribution.
Here, for $T^f_{(U,V),C}$ a constrained translate of size $k$, shift $q$ and power $p$, we have
\[
\EE_{\symm_n}[T^f_{(U,V),C}] = \binom{n-q}{k-q} \overline{f}_C(n)\frac{1}{(n-k)!}
\]
with $\overline{f}_C(n)$ as in Lemma~\ref{l:vincular-poly}.
Moving the term $\frac{1}{(n-k)!}$ to the left side, the result follows.
\end{proof}

\subsection{Examples of Regular Statistics}
\label{ss:regular-example}

In this subsection, we present several previously studied families of statistics that are regular.
The first are \emph{classical pattern counts}.
For $\sigma \in \symm_k$, say that $\pi \in \symm_n$ \emph{contains} $\sigma$ at positions $i_1 < \dots < i_k$ if $\sigma(1) \dots \sigma(k)$ and $\pi(i_1) \dots \pi(i_k)$ have the same relative order and let $N_\sigma(\pi)$ be the number of times $\sigma$ occurs in $\pi$.
Zeilberger showed for $X$ uniformly random in $\symm_n$ that $\EE[N_\sigma^d]$ is a polynomial in $n$ of degree $kd$~\cite{Zeilberger}.
A recent generalization of this result appears in~\cite{DK}, where the authors study a family of pattern counting statistics.

\begin{definition}[{\cite[Def.~2.1]{DK}}]
    \label{def:bivincular}
    A \emph{bivincular pattern}  is a triple $(\sigma,A,B)$ where $\sigma \in \symm_k$ and $A,B \subseteq [k-1]$.
	The partial permutation $(I,J)$ \emph{matches} $(\sigma,A,B)$ 
	if the sequences $I$ and $\sigma(J)$ are increasing,
	 $a \in A$ implies $i_{a+1} = i_a + 1$, and 
	 $b \in B$ implies $j_{\sigma(b+1)} = j_{\sigma(b) + 1}$.
     In essence, $A$ and $B$ provide constraints for $I$ and $J$, respectively.
     
	Given polynomials $f,g \in \RR[x_1,\dots,x_k]$, define a statistic $N^{f,g}_{\sigma,A,B}: \symm_n \rightarrow \RR$ 
	associated to the quintuple $(v,A,B,f,g)$
	to be the weighted 
	enumeration
	\begin{equation}
    \label{eq:bivincular}
	N^{f,g}_{\sigma,A,B} := \sum_{(I,J) \ \mathrm{ matches}\ (\sigma,A,B)} f(I)g(J)\cdot \one_{I,J}.
	\end{equation}
	We call the quintuple $(\sigma,A,B,f,g)$ a {\em weighted bivincular pattern} with \emph{size} $k$, \emph{shift} $|A| + |B|$ and \emph{power} $k + \deg f + \deg g - |A| - |B|$.
    A \emph{bivincular statistic} is a finite sum of $N^{f,g}_{\sigma,A,B}$'s.
\end{definition}

When $A = B = \varnothing$ and $f \equiv g \equiv 1$ in~\eqref{eq:bivincular}, we recover the pattern enumeration statistics
$N_\sigma$ from before. 
The $B = \varnothing$ case is known as a {\em (weighted) vincular} pattern counting, while further setting $f \equiv g \equiv 1$ gives an unweighted vincular pattern count.
In this case, we write $N_{\sigma,A}$ instead of the more cumbersome $N^{1,1}_{v,A,\varnothing}$.

Binvincular statistics are regular.
\begin{proposition}
\label{p:bivincular}
For $(\sigma,A,B,f,g)$ a weighted bivincular pattern, $N^{f,g}_{\sigma,A,B}$ is a regular statistic equivalent size, shift and power.
\end{proposition}

\begin{proof}
To see $N^{f,g}_{\sigma,A,B}$ is a regular, let $(I,J)$ match $(\sigma,A,B)$.
Then there is a packed partial permutation $(U,V)$ with $U \cup V = [m]$ and set $S$ so that $(I,J) = (S(U),S(V))$.
Necessarily $(U,V)$ also matches $(\sigma,A,B)$.
For $a \in A$, we see $u_{a+1} = u_a+1$, and likewise for $b \in B$ we see $v_{b+1} = v_b+1$.
Additionally, $h(x_1,\dots,x_m) = (f(x_{u_1},\dots,x_{u_k})g(x_{v_1},\dots,x_{v_k})$ is a polynomial of degree $\deg f + \deg g$.
Then for $C = \{u_a: a \in A\} \cup \{v_b:b\in B\}$ we see every term in $N^f_{(U,V),C}$ appears in $N^{f,g}_{\sigma,A,B}$.
Subtracting off this term and applying induction, we see $N^{f,g}_{\sigma,A,B}$ is a sum of constrained translates, hence regular.

Clearly the size of $N^{f,g}_{\sigma,A,B}$ is $|\sigma|$.
For the claim about shift, when an indicator function $\one_{I,J}$ indexed by disjoint sets 
$I \cap J = \varnothing$ appears in the indicator expansion of $N^{f,g}_{v,A,B}$, the vincular conditions imposed by $A$ on $I$ and $B$ on $J$ are disjoint.
The claim about power can be seen in a similar way.
\end{proof}

As a corollary, we see Theorem~\ref{t:regular-moment} applies to bivincular statistics.

\begin{corollary}
    \label{c:bivincular-moments}
For $\Gamma = \{(\sigma_i,A_i,B_i,f_i,g_i):i\in I\}$ a finite set of a weighted bivincular patterns with sizes $k_i$, shifts $q_i$ and powers $p_i$ for $i \in I$, let $k = \max k_i$, $q = \max q_i$ and $p = \max p_i$.
Then there exists a polynomial $f_{\Gamma,d} \in \ZZ[n,m_1,\dots,m_k]$ of degree $dp + dq$ where $\deg n = 1$ and $\deg m_i = i$ so that for $\lambda \vdash n$
\[
(n)_{dq}\EE_\lambda\left[\left(\sum_{i \in I} N^{f_i,g_i}_{\sigma_i,A_i,B_i}\right)^d\right] = f_{\Gamma,d}(n,m_1(\lambda),\dots,m_k(\lambda)).
\]
\end{corollary}

\begin{proof}
    This follows immediately from Proposition~\ref{p:bivincular} and Theorem~\ref{t:regular-moment}.
\end{proof}

The special case where the patterns in $\Gamma$ are classical patterns is~\cite{GP}, extending the main result of~\cite{GR} where $\Gamma$ is a single classical pattern.
In this case the term $q=0$ so the term $(n)_{dq}$ vanishes and the $d$th moment equals $f_{\Gamma,d}$.
Their proofs are based on representation theory, and this polynomiality is a necessary consequence of their method.

The main moment result in~\cite{DK} is the analogue of Corollary~\ref{c:bivincular-moments} for the uniform distribution on all of $\symm_n$.
By Corollary~\ref{c:bivincular-moments}, this follows from Theorem~\ref{t:uniform-regular-moment}.

In~\cite{KLY}, the authors introduce a family of statistics analogous to bivincular statistics, but for perfect matchings.
Interpreting perfect matchings as fixed--point--free involutions, these statistics are regular statistics applied to a uniformly random member of $K_{2^{n/2}}$ with $n$ even.
By applying Theorem~\ref{t:moment} for this cycle type we recover their results.
We refer the interested reader to~\cite[\S7.3]{HR} for details.

\section{Asymptotic Results}
\label{sec:Asymptotic}

Theorem~\ref{t:regular-asymptotics} shows for $\Psi$ a sufficiently generic regular statistic and $\{\lambda^{(n)}\}$ a sequence of partitions with $X_n$ uniform on $K_\lambda$ that the limiting behavior of $\EE[\Psi (X_n)]$ only depends on the sequence $\{m_1(\lambda^{(n)})\}$ while $\VV [\Psi(X_n)]$ only depends on $\{m_1(\lambda^{(n)})\}$ and $\{m_2(\lambda^{(n)})\}$.
We will prove these results by a careful analysis of the leading term  in a product of regular statistics.
As an application, we will show for $\{\lambda^{(n)}\}$ with $\lim_{n \to \infty} m_1(\lambda^{(n)})/n = \alpha$ and $X_n$ uniform on $K_{\lambda^{(n)}}$ that there is an associated limiting object called a permuton.
We also extend Hofer's CLT for vincular statistics~\cite{Hofer} to permutations with all long cycles, generalizing a result of Fulman for the descent statistic~\cite{Fulman}.
This last result relies only on Theorem~\ref{t:moment}.

\subsection{Asymptotics of Regular Statistics}

Our main tool for understanding asymptotics of regular statistics makes explicit the observation that, since $\deg m_i = i$, when $i > 1$ such terms cannot contribute to the leading order terms of expected values.

\begin{theorem}
\label{regular-class-asymptotics}
Let $\Psi$ be a regular permutation statistic of size $k$ and power $p$.
Then there exists a polynomial
$g_\Psi \in \RR[x_1, \dots, x_n]$ of degree at most $p$ such that for $\lambda \vdash n$
\begin{equation}
\EE_\lambda\left[ \frac{\Psi}{n^p}\right] = g_\Psi \left( \frac{m_1(\lambda)}{n}, \frac{m_2(\lambda)}{n^2}, \dots, \frac{m_k(\lambda)}{n^k}  \right) + O(n^{-1}).
\end{equation}
\end{theorem}

\begin{proof}
Theorem~\ref{t:regular-moment}
 says
\[
\EE_\lambda \left[ \frac{\Psi}{n^p} \right] = \frac{f_\Psi(n,m_1, \dots, m_k)}{n^{p+q}} \cdot \frac{n^q}{(n)_q}
\]
 where $f_\Psi$ is a polynomial of degree at most $p+q$ under the degree conventions $\deg n = 1$ and $\deg m_i = i$.
 Every monomial in $f$ is of the form $n^{a_0} m_1^{a_1} \cdots m_k^{a_k}$ where the exponents $a_0, a_1, \dots, a_k$ satisfy
 $a_0 + \sum_{i = 1}^k i \cdot a_i = \ell \leq p + q$.
 We write
\[
 \frac{n^{a_0} m_1^{a_1} \cdots m_k^{a_k}}{n^{p+q}} = \left(  \frac{n}{n}  \right) \left(  \frac{x_1}{n}  \right) \left(  \frac{x_2}{n^2}  \right) \cdots 
 \left(  \frac{x_k}{n^k}  \right) \cdot n^{-(p+q-\ell)}.
\]
 Grouping terms of degree $\ell$ into the polynomial $g_\ell$, we have 
\begin{equation}
    \label{eq:moment-polynomial}
 \EE \left[ \frac{\Psi(X)}{n^p} \right] = \frac{n^q}{(n)_q} \cdot \sum_{\ell = 0}^{k+q} n^{-(p+q-\ell) } \cdot
  g_{\ell} \left(\frac{m_1}{n}, \frac{m_2}{n^2}, \dots, \frac{m_k}{n^k} \right)
\end{equation}
 and the result follows by taking $g_\Psi = g_{p+q}$.
\end{proof}

\begin{corollary}
\label{expectation-fixed-point}
	Let $\Psi$ be a regular permutation statistic of power $p$ and $\{\lambda^{(n)}\}$ be a sequence of cycle types so that 
$
\lim_{n \to \infty}m_1(\lambda^{(n)})/n = \alpha \in [0,1].
$
Then exists a polynomial $f \in \mathbb{R}[x]$ so that
\begin{equation}
	\lim_{n \to \infty} \EE_{\lambda^{(n)}} \left[ \frac{\Psi}{n^p}\right] = f(\alpha).
\end{equation}
\end{corollary}

\begin{proof}
	Let $k$ be the size of $\Psi$, define $g_\Psi$ as in Theorem~\ref{regular-class-asymptotics} and observe that $m_i(\lambda^{(n)}) \leq n/i$ for all $i$.
	Then
\[
\lim_{n \to \infty} \EE_{\lambda^{(n)}} \left[\frac{\Psi}{n^p}\right] = \lim_{n \to \infty} g_\Psi \left(\frac{m_1}{n}, \frac{m_2}{n^2}, \dots, \frac{m_k}{n^k}  \right) + O(n^{-1}) = g_\Psi(\alpha,0,\dots,0).
\]
Setting $f(x) = g_\Psi(x,0,\dots,0)$, the result follows.
\end{proof}

Note the polynomial in Corollary~\ref{expectation-fixed-point} can be 0.
For example, the $2$--cycle count $m_2$ is a regular statistic with power $2$ with $g_{m_2} = m_2/n^2$.
Since $m_2(\pi) \leq n/2$ for all $\pi \in \symm_n$, we see $g_{m_2}/n^2 \to 0$.

For a large class of regular statistics including vincular statistics, the normalization factor $n^p$ in Corollary~\ref{expectation-fixed-point} is the correct normalizing coefficient.
For such statistics, we now extend Corollary~\ref{expectation-fixed-point} to a law of large numbers.
To do so, we will show structural properties for the variance of a regular statistic.
First, we show a technical lemma about vincular translates.
Recall $R_n$ from~\eqref{eq:reynolds}.

\begin{lemma}
	\label{translate-product}
Let $T^f_{(I,J),A}$ and $T^g_{(K,L),B}$ be constrained translates with respective powers $p$ and $q$.
Also, let $k,\ell$ be the sizes and $(\mu,\nu),(\rho,\tau)$ be the cycle--path types of $(I,J)$ and $(K,L)$, respectively.
Then the sum of coefficients for indicators $1_{UV}$ with $(U,V) \in \symm_{n,d}$ and cycle--path type  $(\alpha,\beta)$ in
\begin{equation}
\label{eq:translate-covariance}
	T_{(I,J),A}^f \cdot T_{(K,L),B}^g - T_{(I,J),A}^f \cdot R_n (T_{(K,L),B}^g),
\end{equation}
is $O(n^{p+q +d - k -\ell})$  and $O(n^{p+q-1})$ when $|\alpha| + |\beta| = k + \ell$, or equivalently $\alpha = \mu \cup \rho, \beta = \nu \cup \tau$.
\end{lemma}

Implicitly,  Lemma~\ref{translate-product} is an assertion about the naive expansion of Equation~\eqref{eq:translate-covariance} into $1_{IJ}$'s. 

\begin{proof}
We treat the case where  $|\alpha| + |\beta| = k + \ell$ first as this will be the only case of real interest.

Let $\overline{f}$ and $\overline{g}$ be associated to $f$ and $g$ as in Lemma~\ref{l:vincular-poly}.
We first compute the sum of coefficients for indicators whose size is $k+\ell$.
All terms in the product $T_{(I,J),A}^f \cdot T_{(K,L),B}^g$ have the form
	\[
	\one_{(S[I] \cup T[K], S[J] \cup T[L])}
	\]
	where $S \in \binom{[n]}{k}_A$ and $T \in \binom{[n]}{\ell}_B$, and those with cycle--path type of size = $k + \ell$ have $S \cap T = \varnothing$.
	Define
\[
\overline{g}_S(n) := \sum_{T = \{t_1 < \cdots < t_\ell \}\in \binom{[n]}{\ell}_B \cap \binom{[n]-S}{\ell}}  g(t_1,\dots,t_\ell).		
\]
This depends on the set $S$ in a non-trivial way.
However, one can see $\overline{g} - \overline{g}_S$ is of degree at most $q-1$ since $T$ not in $\binom{[n]-S}{\ell}$ necessarily includes one of the $k$ elements of $S$, and the number of such $T$ is a polynomial in $n$ of degree $\ell - 1$.
Therefore, the sum of coefficients for such indicators will be
\[
\sum_{S = \{s_1 < \cdots < s_k\} \in \binom{[n]}{k}_A} f(s_1,\dots,s_k) \overline{g}_S = \sum_{S = \{s_1 < \cdots < s_k \} \in \binom{[n]}{k}_A} f(s_1,\dots,s_k) (\overline{g} + O(n^{q-1})),
\]
which is $ (\overline{f} \cdot \overline{g})(n) + O(n^{p+q-1})$.
	
	We now compute the sum of coefficients for indicator with $|\alpha| + |\beta| = k + \ell$ in the product $T_{(I,J),A}^f \cdot R_n\, (T_{(K,L),B}^g)$, which expands as
	\begin{align*}
T_{(I,J),A}^f \cdot R_n\, (T_{(K,L),B}^g) &= T_{(I,J),A}^f \cdot \overline{g}(n) R_n\,( \one_{(K,L)}) \\
& = \overline{g}(n) \sum_{S \in \binom{[n]}{k}} \one_{(S[I],S[J])} \cdot  \frac{1}{n!} \sum_{\sigma \in \symm_n} \one_{(\sigma(K), \sigma(L))}.
	\end{align*}
	For each partial permutation $(S[I],S[J])$ and permutation $\sigma$ so that the set $\{\sigma(1), \dots, \sigma(\ell)\} \cap S$ is empty, 
	we obtain an indicator whose cycle--path type $(\alpha,\beta)$ satisfies $|\alpha| + |\beta| = k + \ell$, while all other indicators have smaller size.
	There are $(n-k)_\ell \cdot (n-\ell)!$ such permutations $\sigma$ for each $S$.
	Therefore the desired sum of coefficients is
\[
	\overline{f}(n) \cdot \overline{g}(n) \cdot \frac{(n-k)_\ell}{(n)_\ell} = (\overline{f}\cdot \overline{g})(n) + O(n^{p+q-1}),	
\]
	and the result follows for the case where $|\alpha| + |\beta| = k + \ell$.
	
To obtain an indicator with $v = |\alpha| + |\beta| < k + \ell$, note $k+\ell-v$ values in $S$ must be merged with values in $T$, which can only happen in finitely many ways. 
Therefore, a similar argument shows the sum of coefficients for indicators with graph $G$ in either product is at most $O(n^{p+q-v})$ and the result follows.
We do not require further cancellation in this case.
\end{proof}

Using this, we can show the limiting variance of a regular statistic depends only on the proportion fixed points and 2-cycles.

\begin{theorem}
\label{regular-variance}
	Let $\Psi$ be a regular statistic of size $k$ and power $p$.
	Also, let $\{\lambda^{(n)}\}$ be a sequence of integer partitions so that
\[
		\lim_{n \to \infty} \frac{m_1(\lambda^{(n)})}{n} = \alpha \in [0,1], \quad \lim_{n \to \infty} \frac{m_2(\lambda^{(n)})}{n} = \beta \in [0,1].
\]
	Then there exists polynomials $V_{\Psi,1},V_{\Psi,2} \in \CC[x]$ independent so that
	\begin{equation}
	\label{eq:beta-variance}
		\lim_{n \to \infty} \frac{\VV_{\lambda^{(n)}}[\Psi(X_n)]}{n^{2p-1}} = V_{\Psi,1}(\alpha) + \beta V_{\Psi,2}(\alpha).
	\end{equation}
	
\end{theorem}

\begin{proof}
For $n \geq 1$, let $w^{(n)} \in K_{\lambda^{(n)}}$.
By definition 
\[
\VV_{\lambda^{(n)}}\left[\Psi\right] = \EE_{\lambda^{(n)}}[\Psi^2] -\EE_{\lambda^{(n)}}[\Psi]^2 = \EE_{\lambda^{(n)}}[ (\Psi - R_n\, (\Psi))^2].
\]
By Proposition~\ref{p:regular-algebra} we can view $\Xi = (\Psi - R_n\,(\Psi))^2$ as a regular statistic with power at most $2p$.
By~\eqref{eq:moment-polynomial}, there exist polynomials $g_{2p}$ and $g_{2p-1}$ so
\begin{align*}
\frac{\VV_{\lambda^{(n)}}[\Psi]}{n^{2p-1}} &= \EE_{\lambda^{(n)}}\left[\frac{\Xi}{n^{2p-1}}\right] = n \EE_{\lambda^{(n)}}\left[\frac{\Xi}{n^{2p}}\right] \\
	&= n \cdot g_{2p} \left(\frac{m_1}{n},\frac{m_2}{n^2}, \dots, \frac{m_k}{n^k}\right) + g_{2p-1}\left(\frac{m_1}{n},\frac{m_2}{n^2}, \dots, \frac{m_k}{n^k}\right) + O(n^{-1}).
\end{align*}
Recall $g_{2p}$ is homogeneous of degree $2p$.
Therefore if the coefficient of $x_1^{2p}$ in $g_{2p}$ is zero, the result will follow by taking the limit as $n \to \infty$ while setting $V_{\Psi,1}(\alpha) = g_{2p-1}(\alpha,0,\dots,0)$ and $V_{\Psi,2}(\alpha)$ to be the coefficient of $x_2$ in $g_{2p}$ evaluated at $(\alpha,0,\dots,0)$.

We now show $g_{2p}$ has no terms of the form $x_1^{2p}$.
Note $2p$ is the power of $\Xi$ and let $2q$ be its shift.
The above vanishing is equivalent to showing $f_{\Xi,1}$ as in Theorem~\ref{t:regular-moment} has no terms of the form $n^a m_1^b$ with $a+b = 2p + 2q$.
We prove this using Corollary~\ref{c:top-degree-moment} and Lemma~\ref{translate-product}, the former to characterize such terms and the latter to show their degree drops.
Since $\Psi$ is a regular statistic, there is a collection of packed triples $\Upsilon$ so that
	\begin{equation}
		\Psi = \sum_{((I,J),A,f) \in \Upsilon} c_{(I,J),A,f} T^f_{(I,J),A}.
	\end{equation}
	We then have
	\begin{align}
\VV_{\lambda^{(n)}}\left[\Psi\right] = \sum_{((I,J),A,f) \in \Upsilon} \sum_{((K,L),B,g) \in \Upsilon} c_{(I,J),A,f} c_{(K,L),B,g} \mbox{Cov}\left[T^f_{(I,J),A} T^g_{(K,L),B}\right].
	\end{align}
	where for $X,Y$ random variables $\mbox{Cov}[X,Y] = \EE[X^2] -2\EE[XY] + \EE[Y^2]$ is the covariance.
We obtain $\mbox{Cov}_{\lambda^{(n)}}\left[T^f_{(I,J),A}, T^g_{(K,L),B}\right]$ by taking $\EE_{\lambda^{(n)}}$ of Equation~\eqref{eq:translate-covariance}.

Viewing $\mbox{Cov}_{\lambda^{(n)}}\left[T^f_{(I,J),A}, T^g_{(K,L),B}\right]$ as a regular statistic (more accurately, we should apply $R_n$ to~\eqref{eq:translate-covariance}), let $f_{\mathrm{Cov}}$ be the associated polynomial as in Theorem~\ref{t:regular-moment}.
Corollary~\ref{c:top-degree-moment} shows that a term $n^a m_1^b$ with $a + b = 2p + 2q$ can only occur in $f_{\mathrm{Cov}}$ only when $T^f_{(I,J),A}$ and $T^g_{(K,L),B}$ both have power $p$ and the partial permutations $(I,J)$ and $(K,L)$ have cycle--path types $(\varnothing,1^k)$ and $(\varnothing,1^\ell)$, respectively.
This implies $|I \cup J| = 2k$ and $|K \cup L| = 2\ell$.
In this case, by Lemma~\ref{translate-product} the total degree for indicators with cycle--path type $(\varnothing,1^{k+ \ell})$ in the naive expansion of~\eqref{eq:translate-covariance} is $O(n^{2p-1}))$.
This completes our proof.
	\end{proof}

\begin{example}
The excedance statistic $\exc = T^{1}_{(1)(2)}$ has
\begin{equation}
\label{eq:exc}
R_n\,\exc = \frac{n-m_1}{2} \quad \mbox{and} \quad R_n\,\exc^2 - (R_n\,\exc)^2 = \frac{n-m_1 - 2m_2}{12}.    
\end{equation}
Since the power of $\exc$ is 1, this gives $V_{\exc,1}(\alpha) = (1-\alpha)/12$ and $\beta V_{\exc,2}(\alpha) = \beta/6$.
See~\cite[\S~8.4]{HR} for a derivation of~\eqref{eq:exc}.

\end{example}

As a consequence, we derive a weak law of large numbers for regular statistics.
	
\begin{corollary}
	\label{regular-weak-law}
	Let $\Psi$ be a regular statistic with power $p$, $\alpha \in [0,1]$ and $\{\lambda^{(n)}\}$ be a sequence of partitions so that
	\begin{equation}
	\label{eq:fixed-point}
	\lim_{n \to \infty} \frac{m_1(\lambda^{(n)})}{n} = \alpha.
	\end{equation}
For $X_n$ a uniformly random element of $K_{\lambda^{(n)}}$ and $\epsilon > 0$, 
	\begin{equation}
	\lim_{n \to \infty} \PP\left[\left|\frac{1}{n^p}(\Psi(X_n)  - \EE_{\lambda^{(n)}}[\Psi])\right| < \epsilon\right] = 0.
	\end{equation}
\end{corollary}
This is convergence in probability.

\begin{proof}
Combining Theorem~\ref{regular-class-asymptotics} and Theorem~\ref{regular-variance} with Chebyshev's inequality we have
	\begin{equation}
\PP\left( \frac{1}{n^p}\left| \Psi(X_n) - \EE_{\lambda^{(n)}}[\Psi] \right| >  n^{1/3} \sqrt{\VV_{\lambda^{(n)}}\left[\frac{1}{n^p}\Psi\right]} \right) \leq n^{-1/3}.
	\end{equation}
	Since 
	\[
	\lim_{n \to \infty} \VV_{\lambda^{(n)}}\left[\frac{1}{n^p}\Psi\right]  = O(n^{-1}),
	\]
	the result follows.
\end{proof}

\subsection{Convergence for long cycles}

There is an immense literature on limiting distributions for permutation statistics and the limiting behavior of permutations.
We briefly summarize some results relevant to our work. A more extensive discussion of prior results on these behaviors on specific conjugacy classes can be found in~\cite[\S~8.1]{HR}.

Vincular statistics are the most general family of pattern counting permutation statistics for which a central limit theorem is known.

\begin{theorem}
\label{hofer-convergence}
{\em (Hofer \cite[Thm. 3.1]{Hofer})}
Let $v \in \symm_k$, $A \subsetneq [k-1]$ with $|A|=q$ and $X_n = N_{v,A}(\Sigma_n)$ where $\Sigma_n$ is uniform in $\symm_n$.
Then there exists a constant $\sigma_{v,A}>0$ so that as $n \to \infty$ we have
\begin{equation}
\left(  \frac{X_n - \frac{n^{k-q}}{k! (k-q)!}}{n^{k-q-\frac{1}{2}}}
 \right) \xrightarrow{d} \mathcal{N}(0, \sigma^2_{v,A})
\end{equation}
where $\xrightarrow{d}$ is convergence in distribution.
\end{theorem}

The equivalent result with no guarantee that $\sigma^2_{v,A} > 0$ appears in~\cite{Feray}, and another slight generalization appears in~\cite{Janson}.

A sequence of integer partitions  $\{\lambda^{(n)}\}$ with $\lambda^{(n)} \vdash n$ \emph{has no short cycles} if
$
\lim_{n \to \infty} m_i\left(\lambda^{(n)}\right) = 0
$
for all $i$.
Since these quantities are integers, this means $m_i(\lambda^{(n)})$ is eventually $0$.
Fulman showed for $\{\lambda^{(n)}\}$ with no short cycles that the descent statistic $\des$ is asymptotically normal.
We extend this to all vincular statistics using Hofer's central limit theorem.
To do so, we use a simple corollary of Theorem~\ref{t:moment}.

\begin{proposition}
\label{local-statistic-long} 
Fix $d,k \geq 0$ and
let $f:\symm_n \to \mathbb{R}$ be a linear combination of $1_{IJ}$'s with each $(I,J)$ of size at most $k$.
Then $\EE_\lambda[f^d]$ if a function of $m_1(\lambda),m_2(\lambda),\dots,m_{kd}(\lambda)$.
In particular, if these quantities are all 0
\begin{equation}
\EE_\lambda[f^d] = \EE_{(n)}[f^d].
\end{equation}

\end{proposition}

\begin{proof}
For $(I,J)$ and $(K,L)$ partial permutations of sizes $k$ and $\ell$, note $1_{IJ} \cdot 1_{KL}$ vanishes or is $1_{UV}$ where $(U,V)$ is a partial permutation of size at most $k+\ell$.
Therefore, the result will follow from the case where $d=1$.
Write
\[
f = \sum_{(I,J) \in \symm_{n,k}} c_{IJ} \cdot 1_{IJ}.
\]
Then by linearity of expectation
\[
\EE_\lambda[f] = \sum_{(I,J) \in \symm_{n,k}} c_{IJ} \cdot \EE_\lambda[1_{IJ}] = \sum_{(I,J) \in \symm_{n,k}} c_{IJ} \cdot \EE_{(n)}[1_{IJ}] = \EE_{(n)}[f]
\]
with the second equality by Theorem~\ref{t:moment}.
\end{proof}

Note this result holds for regular permutation statistics, but is far more general since it does require any stability of coefficients.
We now extend Theorem~\ref{hofer-convergence} to cycle types with all long cycles.

\begin{theorem}
\label{our-long-cycles}
Let $v \in \symm_k$ and $A \subsetneq [k-1]$ satisfy $|A| = q$.
For $\{ \lambda^{(n)} \}$ a sequence of partitions of $n$ with no short cycles and $\Sigma_n$ uniform on $K_{\lambda^{(n)}}$, there exists $\sigma_{v,A} > 0$ so that 
\begin{equation}
\left(
\frac{N_{v,A}(\Sigma_n) - \frac{n^{k-q}}{k! (k-q)!}}{n^{k-q-\frac{1}{2}}}
\right)  \xrightarrow{d} \mathcal{N}(0, \sigma_{v,A}^2).
\end{equation}
\end{theorem}

\begin{proof}
One first establishes the result for $\lambda = (n)$, which is straightforward since $\Sigma_n$ differs from a uniformly random permutation in $\symm_n$ in finitely many positions with high probability. 
Since $\{ \lambda^{(n)} \}$ has all long cycles, Proposition~\ref{local-statistic-long} guarantees for $d \geq 1$ that 
\[
\EE_{\lambda^{(n)}}[N_{v,A}^d] = \EE_{(n)}[N_{v,A}^d]
\]
for $n$ sufficiently large.
The result now follows from the Method of Moments.
\end{proof}

\subsection{Permutons} A sequence $\{ w^{(n)} \}$ of permutations with $w^{(n)} \in \symm_n$ is called {\em quasi-random} if for all $k \geq 1$ and for all $v \in \symm_k$ we have
\begin{equation}
\label{quasi-random-def} \frac{N_v(w^{(n)})}{\binom{n}{k}} = \frac{1}{k!} + o(1).
\end{equation}
This pattern-theoretic definition of randomness asserts that $\{ w^{(n)} \}$, in the limit, contains equal proportion of patterns of size $k$
for every $k \geq 1$.
This property shows that the small-scale behavior of a quasi-random sequence is close to that for a sequence of uniformly random permutations.
Quasi-random sequences of permutations were studied by Cooper \cite{Cooper} in his PhD thesis at UC San Diego.
Answering a question of Graham (see \cite{Cooper}), Kr\'al' and Pikhurko proved \cite[Thm. 1]{KP} the surprising result that 
Equation~\eqref{quasi-random-def} need only hold for $v \in \symm_4$, but that checking patterns of size up to 3 is insufficient.

\begin{corollary}
\label{quasirandom-permutation-generation}
Let $\mu_n$ be a sequence of congugacy-invariant probability distributions on $\symm_n$.
For $\Sigma_n \sim \mu_n$, if
\begin{equation}
\PP \left(  \lim_{n \to \infty} \frac{m_1(\Sigma_n)}{n} = 0 \right) = 1,
\end{equation}
then the sequence $\Sigma_n$ is quasi-random.
\end{corollary}

\begin{proof}
This follows immediately from Corollary~\ref{expectation-fixed-point} with $\alpha = 0$.
\end{proof}

A result analogous to Corollary~\ref{quasirandom-permutation-generation} is possible when the limiting proportion of fixed points tends to any $\alpha \in [0,1]$.
To state this, we need the following definition, introduced in~\cite{HKMRS}.
A \emph{permuton} is a measure $\mu$ on the Borel $\sigma$-algebra of $[0,1] \times [0,1]$ so that for every measurable $A \subseteq [0,1]$
\[
\mu(A \times [0,1]) = \mu([0,1] \times A) = \mathcal{L}_{[0,1]}(A)
\]
where $\mathcal{L}$ is Lebesgue measure.
Given a permuton $\mu$, sample $(X_1,Y_1),\dots,(X_k,Y_k) \sim \mu$ i.i.d., reordereed so that $X_1 < \dots < X_k$, and define $\Pi^\mu_k$ to be the permutation with the same relative order as $(Y_1,\dots,Y_k)$.
A sequence $\{\mu_n\}$ of measures on $\symm_n$ \emph{converges} to $\mu$, denoted $\mu_n \to \mu$, if given $\Sigma_n \sim \mu_n$, for every $k$ and $v \in \symm_k$ we have
\[
\lim_{n \to \infty} \frac{N_v(\Sigma_n)}{\binom{n}{k}} = \PP(\Pi^\mu_k = v).
\]

\begin{theorem}
	\label{alpha-permuton}
Let $\mu_n$ be a sequence of congugacy-invariant probability distributions on $\symm_n$ and $\alpha \in [0,1]$ so that for $\Sigma_n \sim \mu_n$, we have 
\[
\PP\left(\lim_{n \to \infty} \frac{m_1(\Sigma_n)}{n} = \alpha \right) = 1.
\]
Then $\mu_\alpha$ so $\mu_n \to \mu_\alpha$ with $\mu_\alpha = (1-\alpha) U + \alpha I$ where $U$ is Lebesgue measure on $[0,1] \times [0,1]$ and $I$ is Lebesgue measure on $\{(x,x): x\in [0,1]\}$. 
\end{theorem}

\begin{proof}
	The existence of $\mu_\alpha$ follows from Corollary~\ref{expectation-fixed-point} and~\cite[Thm. 1.6]{HKMRS}.
	The explicit description of $\mu_\alpha$ can be seen as follows.
	Let $\mu_{\alpha,n}$ be the measure on $\symm_n$ given by sampling $\lfloor \alpha n \rfloor$ values in $[n]$ uniformly at random, setting those as fixed points, then permuting the remaining values uniformly at random.
	It is clear both that $\mu_{\alpha,n} \to \mu_\alpha$ and that $\mu_{\alpha,n}$ satisfies the assumptions of our theorem.
\end{proof}

The explicit description of $\mu_\alpha$ in Theorem~\ref{alpha-permuton} is due to Valentin F\'eray, and we thank him for allowing us to include his proof.

\section{Final Remarks}
\label{sec:Conclusion}

\subsection{Other groups}
 In this paper we consider functions on the symmetric group.
One can also consider functions on other groups $G$ from a 
 Fourier theoretic perspective (the cases $G = S^1$ and $G = (\ZZ/2\ZZ)^n$ were mentioned in the introduction).
 For representation-theoretic purposes, it is best to consider complex-valued functions on $G$.
 
 \begin{problem}
 \label{other-groups-problem}
 Extend the results of this paper to (complex-valued) functions $f: G \rightarrow \CC$
 on a wider class of groups $G$.
 \end{problem}
 
One class of groups $G$ to which Problem~\ref{other-groups-problem} could be applied are the complex reflection groups.
Given parameters $n, r,$ and $p$ with $p \mid r$, let $G(r,p,n)$ be the group of $n \times n$ complex matrices $A$ with a unique nonzero entry in every 
row and column such that 
\begin{itemize}
\item every nonzero entry of $A$ is a $r^{th}$ root-of-unity, and
\item the product of the nonzero entries of $A$ is a $(r/p)^{th}$ root-of-unity.
\end{itemize}
We have $G(1,1,n) = \symm_n$ and we may identify $G(r,1,n) = (\ZZ/r\ZZ) \wr \symm_n$ with a wreath product, which have been previously studied from a representation-theoretic perspective~\cite{Rockmore} with applications to image processing~\cite{FMRHO}.
A result analogous to Proposition~\ref{local-statistic-long} has been proved for these groups in~\cite{LLLSY2}.

Another natural class of groups are compact simple Lie groups, over $\CC$ or over finite fields.

\subsection{Joint distributions}
 In this paper we exclusively considered real-valued permutation statistics of the form $\Psi:\symm_n \rightarrow \RR$.
It is natural to ask whether our methods can extend to functions of the form $f:\symm_n \rightarrow \RR^k$.
Since $f = (f_1,\dots,f_k)$ with $f_i : \symm_n \to \RR$, this would correspond to understanding the joint distribution of $f_1,\dots,f_k$.

\begin{problem}
\label{multivariate-problem}
	Use our methods to understand the joint distribution of regular functions.
\end{problem}

In the uniform case, the joint distribution of ordinary pattern counting statistics is relatively well understood.
Building on work of Burstein and H\"ast\"o~ for the bivariate case\cite{BH}, Janson, Nakamura and Zeilberger showed $(N_\sigma)_{\sigma \in \symm_k}$ is jointly normal and explicitly computed its covariance matrix~\cite{JNZ}.
Even-Zohar gives refinements of their work by describing some geometric  properties for the image of this $k!$-variate function~\cite{EZ}.

\subsection{Limiting distributions}
\label{ss:normality}

Theorem~\ref{regular-variance} established a quite general characterization of the asymptotic mean and variance for regular statistics.
However, we do not know how to extend it to higher moments.
As such, we cannot apply the method of moments to establish asysmptotic normality results, as we do for the all long cycles case in Theorem~\ref{our-long-cycles}.
However convergence to normal distributions should be much more general.

\begin{problem}
	Given a sequence $\{\lambda^{(n)}\}$ of cycle types, identify conditions on small cycle counts for which regular statistics are asymptotically normal.
\end{problem} 

Using our work, F\'eray and Kammoun have solved this problem for ordinary pattern counts.
Specifically, they demonstrate for $\{\lambda^{(n)}\}$ a sequence of partitions with $\alpha = \lim_{n \to \infty} m_1(\lambda^{(n)})/n$ and $\beta = \lim_{n \to \infty} m_2(\lambda^{(n)})/n$ with $\alpha < 1$ that any vincular statistics is asympotically normal or degenerate~\cite[Thm~1.3]{FK}.
For classical patterns, they establish non-degeneracy~\cite[Thm~1.4]{FK}.
We remark that their methods extend to joint distributions.

Our proofs of convergence to normality ultimately relied on Hofer's Theorem~\ref{hofer-convergence} showing that the statistics $N_{v,A}$ 
  converge to normality on all of $\symm_n$, or alternatively on~\cite{Janson}.
The F\'eray-Slim proofs rely critically on Hofer's methods as well.

  \begin{problem}
  \label{hofer-problem}
  Use the techniques of this paper to prove Hofer's Theorem~\ref{hofer-convergence}.
  More generally, use the methods of this paper to describe further families of regular permutation statistics that are asymptotically normal.
  \end{problem}
  
Hofer's proof relies on dependency graphs, and ultimately reduces to showing the variance of a regular statistic, as normalized in Theorem~\ref{regular-variance}, does not vanish.
The F\'eray--Kammoun work also relies on this method.
Alternatively, a na\"ive approach to Problem~\ref{hofer-problem} via the Method of Moments would identify the expansion of $N_{v,A}^d$ into constraint translates.
While the $T$-expansions of the statistics $N_{v,A}^d$ become very complicated as $d$ grows, it is plausible the leading order terms are easy to identify.

It would also be interesting to find the limiting distributions of regular statistics that are not asymptotically normal on a sequence of cycle types.
In~\cite{CLP}, the authors study a family of bivincular patters called \emph{very tight}, which converge to a Poisson distribution for patterns of size $k = 2$.
When $k > 2$, the probability to find a very tight pattern of size $k$ in a permutation in $\symm_n$ tends to zero.
Characterizing the possible limiting distributions for vincular statistics in general would be an interesting challenge.
Towards this end, unpublished work of Coopman uses similar methods to identify the unweighted constrained translates that are asymptotically normal~\cite{Coopman}.

\subsection{Other combinatorial objects}
 This paper was exclusively concerned with functions $f: \symm_n \rightarrow \CC$ constructed as linear combinations of indicators for local behavior.
 
 \begin{problem}
 \label{other-objects-problem}
 Extend this approach to the study of statistics $f: \Omega_n \rightarrow \CC$ on other classes $\Omega_n$ of combinatorial objects.
 \end{problem}
 
 For example, at the end of Section~\ref{ss:regular-example} we mention the case  $\Omega_n = \MMM_n$ of perfect matchings by means of the identification 
 $\MMM_n = K_{2^n}$ with fixed-point-free involutions.
One could also define indicator functions for local behavior on the class $\Omega_n = \AAA^n$ of all length $n$ words over a given alphabet $\AAA$ or
 (following Chern-Diaconis-Kane-Rhoades \cite{CDKR1, CDKR2}) the class $\Omega_n = \Pi_n$ of set partitions of $[n]$.

 For $r$ fixed, another setting to which Problem~\ref{other-objects-problem} can be applied is the class
 $\Omega_n = \VVV_{n,r}$ of $r$-uniform hypergraphs on $[n]$.
 A hypergraph $\gamma \in \VVV_{n,r}$ can be thought
 of as Boolean functions $\gamma: \{0,1\}^{\binom{n}{r}} \rightarrow \{0,1\}$ on the vertex set of a hypercube of dimension ${n \choose r}$.
Here, a natural collection of indicators have been introduced and studied from a representation theoretic perspective in~\cite{RSST}.
However, their results are analogous to Proposition~\ref{local-statistic-long} and do not explain the precise dependence of hypergraph statistics on small subgraph counts.

\bibliographystyle{amsalpha}
\bibliography{local-bib}

\newcommand{\etalchar}[1]{$^{#1}$}
\providecommand{\bysame}{\leavevmode\hbox to3em{\hrulefill}\thinspace}
\providecommand{\MR}{\relax\ifhmode\unskip\space\fi MR }
\providecommand{\MRhref}[2]{%
  \href{http://www.ams.org/mathscinet-getitem?mr=#1}{#2}
}
\providecommand{\href}[2]{#2}
\begin{thebibliography}{HKM{\etalchar{+}}13}

\bibitem[B\']{Bona}
M.~B\'{o}na, \emph{The copies of any permutation pattern are asymptotically normal}, Preprint, 2007. {\tt arXiv:0712.2792}.

\bibitem[BH10]{BH}
A.~Burstein and P.~H\"ast\"o, \emph{Packing sets of patterns}, Eur. J. Comb. \textbf{31} (2010), no.~1, 241--253.

\bibitem[CDKR14]{CDKR1}
B.~Chern, P.~Diaconis, D.~Kane, and R.~Rhoades, \emph{Closed expressions for averages of set partition statistics}, Res. Math. Sci. \textbf{1} (2014), no.~1, 1--32.

\bibitem[CDKR15]{CDKR2}
\bysame, \emph{Central limit theorems for some set partition statistics}, Adv. Appl. Math. \textbf{70} (2015), 92--105.

\bibitem[CLP06]{CLP}
S.~Corteel, G.~Louchard, and R.~Pemantle, \emph{Common intervals in permutations}, Discrete Math. Theor. Comput. Sci. \textbf{8} (2006), no.~1, 189--216.

\bibitem[Coo]{Coopman}
Michael Coopman, \emph{Private communication}.

\bibitem[Coo04]{Cooper}
J.~Cooper, \emph{Quasirandom permutations}, J. Combin. Theory Ser. A \textbf{106} (2004), no.~1, 123--143.

\bibitem[DK24]{DK}
Stoyan Dimitrov and Niraj Khare, \emph{Moments of permutation statistics and central limit theorems}, Advances in Applied Mathematics \textbf{155} (2024), 102650.

\bibitem[Dub24]{Dubach}
Victor Dubach, \emph{A geometric approach to conjugation-invariant random permutations}, arXiv preprint arXiv:2402.10116 (2024).

\bibitem[EZ20]{EZ}
C.~Even-Zohar, \emph{Patterns in random permutations}, Combinatorica \textbf{40} (2020), 775--804.

\bibitem[F\'13]{Feray}
V.~F\'{e}ray, \emph{Asymptotic behavior of some statistics in ewens random permutations}, Electron. J. Probab. \textbf{18} (2013), no.~76, 1--32.

\bibitem[FK23]{FK}
Valentin F{\'e}ray and Mohamed~Slim Kammoun, \emph{Asymptotic normality of pattern counts in conjugacy classes}, arXiv preprint arXiv:2312.08756 (2023).

\bibitem[FKL22]{FKL}
J.~Fulman, G.~Kim, and S.~Lee, \emph{Central limit theorem for peaks of a random permutation in a fixed conjugacy class of ${S}_n$}, Ann. Comb. \textbf{26} (2022), 97--123.

\bibitem[FMR{\etalchar{+}}00]{FMRHO}
R.~Foote, G.~Mirchandani, D.N. Rockmore, D.~Healy, and T.~Olson, \emph{A wreath product group approach to signal and image processing .i. multiresolution analysis}, IEEE Trans. Sign. Proc. \textbf{48} (2000), no.~1, 102--132.

\bibitem[Ful98]{Fulman}
J.~Fulman, \emph{The distribution of descents in fixed conjugacy classes of the symmetric groups}, J. Combin. Theory Ser. A \textbf{84} (1998), no.~2, 171--180.

\bibitem[GP]{GP}
C.~Gaetz and L.~Pierson, \emph{Positivity of permutation pattern character polynomials}, Preprint, 2022. {\tt arXiv:2204.10633}.

\bibitem[GR21]{GR}
C.~Gaetz and C.~Ryba, \emph{Stable characters from permutation patterns}, Selecta Math. \textbf{27} (2021), no.~4, 70.

\bibitem[HKM{\etalchar{+}}13]{HKMRS}
C.~Hoppen, Y.~Kohayakawa, C.G. Moreira, B.~R\'ath, and R.M. Sampaio, \emph{Limits of permutation sequences}, J. Combin. Theory Ser. B \textbf{103} (2013), no.~1, 93--113.

\bibitem[Hof18]{Hofer}
L.~Hofer, \emph{A central limit theorem for vincular permutation patterns}, Discrete Math. Theor. Comput. Sci. \textbf{19} (2018), no.~2, 1--26.

\bibitem[HR22]{HR}
Zachary Hamaker and Brendon Rhoades, \emph{Characters of local and regular permutation statistics}, arXiv preprint arXiv:2206.06567 (2022).

\bibitem[HR25]{HR2}
\bysame, \emph{Partial permutations and character evaluations}, preprint (2025).

\bibitem[Hul14]{Hultman}
A.~Hultman, \emph{Permutation statistics of products of random permutations}, Adv. Appl. Math. \textbf{54} (2014), 1--10.

\bibitem[Jan23]{Janson}
Svante Janson, \emph{Asymptotic normality for $m$-dependent and constrained ${U}$-statistics, with applications to pattern matching in random strings and permutations}, Adv. in App. Prob. (2023), 1--54.

\bibitem[JNZ15]{JNZ}
S.~Janson, B.~Nakamura, and D.~Zeilberger, \emph{On the asymptotic statistics of the number of occurrences of multiple permutation patterns}, J. Comb. \textbf{6} (2015), no.~1, 117--143.

\bibitem[KL20]{KL1}
G.~Kim and S.~Lee, \emph{Central limit theorem for descents in conjugacy classes of ${S}_n$}, J. Combin. Theory Ser. A \textbf{169} (2020), 13.

\bibitem[KLY17]{KLY}
N.~Khare, R.~Lorentz, and C.~Yan, \emph{Moments of matching statistics}, J. Comb. \textbf{8} (2017), no.~1, 1--27.

\bibitem[KP13]{KP}
D.~Kr\'al' and O.~Pikhurko, \emph{Quasirandom permutations and characterized by 4-point densities}, Geom. Funct. Anal. \textbf{23} (2013), 570--579.

\bibitem[LLL{\etalchar{+}}23]{LLLSY2}
Jesse~Campion Loth, Michael Levet, Kevin Liu, Sheila Sundaram, and Mei Yin, \emph{Colored permutation statistics by conjugacy class}, arXiv preprint arXiv:2305.11800 (2023).

\bibitem[Roc95]{Rockmore}
D.N. Rockmore, \emph{Fast fourier transforms for wreath products}, App. Comp. Harm. Anal. \textbf{2} (1995), no.~3, 279--292.

\bibitem[RSST18]{RSST}
A.~Raymond, J.~Saunderson, M.~Singh, and R.~Thomas, \emph{Symmetric sums of squares over $k$-subset hypercubes}, Math. Prog. \textbf{167} (2018), no.~2, 315--354.

\bibitem[Zei04]{Zeilberger}
Doron Zeilberger, \emph{Symbolic moment calculus {I}: foundations and permutation pattern statistics}, Annals of Combinatorics \textbf{8} (2004), 369--378.

\end{thebibliography}

\end{document}